\documentclass[review]{siamart1116}

\usepackage{lipsum}
\usepackage{amsfonts}
\usepackage{graphicx}
\usepackage{epstopdf}
\usepackage{algorithmic}

\numberwithin{theorem}{section}

\newcommand{\TheTitle}{Optimal Actuator Design for  Semi-linear Systems} 
\newcommand{\TheAuthors}{M. S. Edalatzadeh, K. A. Morris}

\headers{\TheTitle}{\TheAuthors}

\title{{\TheTitle}}

\author{
  M. Sajjad Edalatzadeh\thanks{Department of Applied Mathematics, University of Waterloo, Waterloo, ON, Canada
    (\email{msedalat@uwaterloo.ca}).}
  \and
  Kirsten A. Morris\thanks{Department of Applied Mathematics, University of Waterloo, Waterloo, ON, Canada (\email{kmorris@uwaterloo.ca}).}}

\usepackage{amsopn}

\usepackage{float}
\usepackage{subcaption}
\usepackage{ifthen}
\newboolean{PrintVersion}
\setboolean{PrintVersion}{false}
\usepackage{mathtools}
\usepackage[framed]{mcode}
\usepackage{comment}
\usepackage[utf8]{inputenc}
\usepackage[T1]{fontenc}
\usepackage{array}
\usepackage{tabu}
\usepackage{colortbl}
\usepackage{mathrsfs}
\usepackage{enumitem}

\newtheorem{assumption}{Assumption}

\newtheorem{remark}{Remark}
\usepackage{cleveref}

\newcommand{\mc}[1]{\mathcal{#1}}

\newcommand{\norm}[2]{\left\| #1 \right\|_{#2} }
\newcommand{\normm}[1]{\left\| #1 \right\| }

\newcommand{\inn}[2]{ \left\langle #1, #2 \right\rangle}

\newcommand{\p}[3]{\frac{\partial^{#3}#1}{\partial#2^{#3}}}

\usepackage{bm}
\newcommand{\xb}{{\bm{x}}}
\newcommand{\yb}{{\bm{y}}}
\newcommand{\eb}{{\bm{e}}}
\newcommand{\pb}{{\bm{p}}}
\newcommand{\ub}{{\bm{u}}}
\newcommand{\rb}{{\bm{r}}}
\newcommand{\hb}{{\bm{h}}}
\newcommand{\fb}{{\bm{f}}}
\newcommand{\T}{{\mathcal{T}}}
\newcommand{\F}{{\mathcal{F}}}
\newcommand{\A}{{\mathcal{A}}}
\newcommand{\B}{{\mathcal{B}}}

\renewcommand{\ss}{{\mathbb{X}}}
\newcommand{{\cs}}{{\mathbb{U}}}
\newcommand{{\as}}{{\mathbb{K}}}

\begin{document}
\nolinenumbers
\maketitle

\begin{abstract}
Actuator location and design are important choices  in controller design for distributed parameter systems. Semi-linear partial differential equations model a wide spectrum of physical systems with distributed parameters. It is shown that  under certain conditions on the nonlinearity and the cost function, an optimal control input together with an optimal actuator choice exists. First-order necessary optimality conditions are derived. The results are applied to optimal actuator and controller design in a nonlinear railway track model {as well as semi-linear wave models.}
\end{abstract}

\begin{keywords}
Actuator design, Semi-linear partial differential equations, Optimal control, Flexible structures, Wave equation
\end{keywords}

\begin{AMS}
49J27, 49K27, 49J20, 49J50, 35L71
\end{AMS}

\section{Introduction} 
Actuator location and design are important design variables  in controller synthesis for distributed parameter systems. 
Finding the best actuator location to control a distributed parameter system can significantly reduce the cost of the control and improve  its effectiveness; see for example, \cite{FahDem,morris2015comparison,Mnoise}.
The optimal actuator location problem has been discussed by many researchers in various contexts; see \cite{frecker2003recent,van2001review} for a review of applications and
\cite{PTZ2013} for optimal location of actuators to maximize controllability in the wave equation.
In  \cite{morris2011linear}, it was proved that an optimal actuator location exists for linear-quadratic control.  Conditions under which using approximations in optimization yield the optimal location are also established.
 Similar results have been obtained for $H_2$ and $H_\infty$ controller design objectives with linear models \cite{kasinathan2013h,DM2013}. Results for optimal design  of linear PDE's have been obtained \cite{kalise2017optimal,MorrisVestCDC}. There are also results on the related problem of optimal sensor location for linear PDE's; see  \cite{Privatetal2013} for locations of maximum observability in the wave equation and
 \cite{WuJacob2015,ZhangMorris} for concurrent sensor choice/estimator design to minimize the error variance.
 
Nonlinearities can  have a significant effect on dynamics, and such systems cannot be accurately modelled by linear differential equations. 
Optimal control  of systems modelled by nonlinear partial differential equations (PDE's) has been studied  for a number of applications, including  wastewater treatment systems  \cite{martinez2000},  steel cooling plants \cite{unger2001},  oil extraction through a reservoir \cite{li2003}, solidification models in metallic alloys \cite{boldrini2009}, thermistors \cite{homberg2010optimal}, Schl\"ogl model \cite{buchholz2013,casas2013}, {FitzHugh–Nagumo system \cite{casas2013}, static elastoplasticity \cite{reyes2016}, type-II superconductivity \cite{yousept2017optimal}, Fokker-Planck equation \cite{fleig2017},  Schr\"odinger equation with bilinear control \cite{ciaramella2016}, Cahn-Hilliard-Navier-Stokes system \cite{hintermuller2017optimal}, wine fermentation process \cite{merger2017optimal}, time-dependent Kohn-Sham model \cite{sprengel2018investigation}, elastic crane-trolley-load system \cite{kimmerle2018optimal}.}
A review of PDE-constrained optimization theory can be found in the books \cite{hinze2008optimization,leugering2012constrained,troltzsch2010optimal}.
State-constrained optimal control of PDEs has also been studied. In \cite{bergounioux2003structure}, the authors investigated the structure of Lagrange multipliers for state constrained optimal control problem of linear elliptic PDEs.  {Research on optimal control of PDEs, such as \cite{casas1997pontryagin,raymond1999hamiltonian}, has focused on parabolic models of partial differential equations with certain structures.
Optimal control of differential equations in abstract spaces has rarely been discussed \cite{meyer2017optimal}. This paper extends  previous results to abstract differential equations without  an assumption of stability. }

Optimal actuator location  has  been addressed for some applications modeled by nonlinear  distributed parameter systems
using  a finite dimensional approximation of the original partial differential equation model.
In \cite{antoniades2001integrating}, authors investigated the optimal actuator and sensor location problem for  a transport-reaction process using  a finite-dimensional model. Similarly, in \cite{lou2003optimal}, the optimal actuator and sensor location of Kuramoto-Sivashinsky equation was studied using a finite-dimensional approximation. Other research concerned with optimal actuator location in problems with nonlinear distributed parameter dynamics can be found in \cite{armaou2008robust,moon2006finite,saviz2015optimal}. 
 To our knowledge, there are no theoretical results on optimal actuator location of nonlinear PDE's.

Theory for concurrent optimal control and actuator design of  a class of controlled semi-linear PDE's is described in this paper. 
The research described  extends previous work on optimal control of PDE's in that the linear part of the partial differential equation is not constrained to be the generator of an analytic semigroup. The input operator of the system is parametrized by the possible actuator designs.  A general class of  PDE's with weakly continuous nonlinear part is considered.
Optimality equations explicitly characterizing the optimal control and actuator are obtained.

Location of actuators on flexible structures has been one of the motivators for research into optimal actuator location \cite{frecker2003recent}. 
Various  models have been   studied. Classical results in the literature concern control of  linear and nonlinear Euler Bernoulli and Timoshenko beam models  \cite[e.g.]{kim1987boundary,lagnese1991uniform}. In recent years, non-classical models of flexible beams such as micro-beam models have also attracted attention \cite[e.g.]{edalatzadeh2016boundary}.
In nonlinear flexible structures, the nonlinearity typically  is  on deformations, not on the rate of deformations. The space in which deformations evolve is compactly embedded in that of rate of deformations. As a result, the nonlinear terms are  weakly continuous in the underlying state space. 
One application of the results in this paper is to the development of an optimal control strategy for the vibration suppression of railway tracks \cite{ansari2011,dahlberg2002dynamic,edalatzadeh2019stability}. 
 {The theory is also illustrated by application to concurrent optimal  control and actuator design for  semi-linear waves  in two space dimensions.}

The paper is organized as follows. After a short paragraph on notation, the problem definition as well as the main results are stated in section 2. Section 3 discusses the existence of a solution to the semi-linear partial differential equation. The existence  of an optimizer  is established in section 4. First-order necessary condition for the optimizer are provided in section 5. In section 6 and 7, the results of the previous sections are applied to the  railway track model and semi-linear wave models, respectively.

\section*{Notation}
Throughout this paper, the letters $c$, $t$, and $\xi$ denote a generic positive constant, temporal variable, and spatial variable, respectively. The blackboard letters as in $\ss$ denote Banach spaces, the calligraphic letters as in $\mc{A}$ denote operators on a Banach space. If an operator is nonlinear its argument is shown in parenthesis as in $\F(\cdot)$. The bold letters as in $\xb$ refer to states evolving in a Banach space; the rest of letters represent physical or generic constants. The adjoint of an operator is denoted by $\mc{A}^*$. The superscript  ${\cdot}^o$ shows that a state or an input is optimal, and the tilde overscript $\tilde{\cdot}$ is reserved for the state of a linearized system unless otherwise stated. Norms and inner products on the underlying state space are typed without any subscript, but on any other spaces, they are shown with a suitable subscript to avoid confusion. {The norm on $L^p(0,\tau;\cs)$ is denoted by $\norm{\cdot}{p}$.} Strong convergences on a Banach space are shown by $\to$, whereas a weak convergence is shown by $\rightharpoonup$. If the Banach space $\ss_1$ is continuously embedded in $\ss_2$, we write $\ss_1\hookrightarrow \ss_2$.
The Banach space $C([0,\tau];\ss)$  will often be indicated by $C(0,\tau;\ss)$ for simplicity of notation. 

\section{Main Results}\label{sec:main results}
Consider a semi-linear  system with state  $\xb(t)$ on a separable reflexive Banach space $\ss$:
\begin{equation}\label{IVP}
\dot{\xb}(t)=\mc{A}\xb(t)+\F(\xb(t))+\mc{B}(\rb)\ub(t), \quad \xb(0)=\xb_0\in \ss,
\end{equation}
The function $\ub(t)$ is the input to the system, and takes values in a reflexive Banach space $\cs$. The control operator $\mc B(\rb)$ depends on a parameter $\rb$ that takes values in a set $K_{ad}$ in a topological space $\mathbb{K}$. The parameter $\rb$  typically has interpretation as possible  actuator designs.
The operators $\mc{A}$, $\F(\cdot)$, and $\mc{B}(\cdot)$  satisfy the following assumptions.
\begin{assumption}\label{A}\leavevmode
\begin{enumerate}
\item \label{as:A1} The state operator $\mc{A}$ with domain $D(\mc{A})$ generates a strongly continuous semigroup $\mc{T}(t)$ on $\ss$.
\item \label{as:A2} {Let $\F(0)=0$}; the nonlinear operator $\F(\cdot)$ is locally Lipschitz continuous on $\ss$; that is, for every positive number $\delta$, there exists $L_{\F\delta} >0$ such that
\begin{equation}
	\normm{\F(\xb_2)-\F(\xb_1)}\le L_{\F\delta} \normm{\xb_2-\xb_1}, \notag
\end{equation}
for all $\normm{\xb_2}\le \delta$ and $\normm{\xb_1}\le \delta$.
\item \label{as:A3} For each $\rb\in K_{ad}$, the input operator $\mc{B}(\rb)$ is a linear bounded operator that maps the input space $\cs$ into the state space $\ss$. This family of operators is uniformly bounded over $K_{ad}$, i.e., there exist a positive number $M_\mc{B}$ such that $\norm{\mc{B}(\rb)}{\mc{L}(\cs,\ss)}\le M_\mc{B}$ for all $\rb\in K_{ad}$.
\end{enumerate}
\end{assumption}

In some cases, due to lack of regularity of the input $\ub$, a classical solution to (\ref{IVP}) is not assured.
\begin{definition}\label{defn-mild}
If $\xb\in C(0,\tau;\ss)$ satisfies
\begin{equation}
\xb(t)=\mc{T}(t)\xb_0+\int_0^t \mc{T}(t-s)\F(\xb(s))ds+\int_0^t \mc{T}(t-s)\mc{B}(\rb)\ub(s)ds, \label{eq:mild solution}
\end{equation} 
for every $\xb_0\in \ss$,
it is said to  be a {\em mild solution} to (\ref{IVP}).
\end{definition}

In \Cref{sec:existence of a solution}, the existence of a unique mild solution to the initial value problem (IVP) (\ref{IVP}) is proven {for $\ub(t)$ in the set
$$U_{ad}=\lbrace \ub\in L^p(0,\tau;\cs): \, \norm{\ub}{p}\le R \rbrace,$$
where $1<p<\infty$.}

\cref{thm:mild}: \textit{Under assumption \ref{A}, for each $\xb_0\in \ss$  and positive number $R$,  there exists $\tau>0$ such that (\ref{IVP}) admits a unique local mild solution $\xb\in C(0,\tau;\ss)$ for all {$\ub\in U_{ad}$}, and all $\rb\in K_{ad}$.}%

For functionals $\phi(\xb)$ on $\ss$ and $\psi(\ub)$ on $\cs$, consider the cost function \begin{equation}
J(\ub,\rb;\xb_0)=\int_0^{\tau} \phi(\xb(t))+\psi(\ub(t))\, dt. \notag \label{eq:cost}
\end{equation}
The optimization problem is to minimize $J(\ub,\rb;\xb_0)$ over all admissible control inputs $\ub\in U_{ad}$, and also over all admissible actuator designs $\rb\in K_{ad}$, subject to (\ref{IVP}) with a fixed initial condition $\xb_0\in\ss$. That is,
\begin{equation}
\left\{ \begin{array}{ll}
\min&J(\ub,\rb;\xb_0)\\
\text{s.t.}&\dot{\xb}(t)=\mc{A}\xb(t)+\F(\xb(t))+\mc{B}(\rb)\ub(t), \quad \forall t\in(0,\tau],\\
&\xb(0)=\xb_0,\\
& \ub\in U_{ad},\\
&\rb\in K_{ad}.
\end{array} \right. \tag{P} \label{eq:optimal problem}
\end{equation}
To guarantee the existence of a unique optimizer, further assumptions are needed on the operators $\F(\cdot)$, $\mc{B}(\cdot)$, the set $K_{ad}$, and the cost function $J(\ub,\rb;\xb_0)$. 
\begin{assumption}\leavevmode \label{B}
\begin{enumerate}
\item \label{as:B1} {Let $\xb_n(t)$ be a bounded sequence in $C(0,\tau;\ss)$ such that $\xb_n\rightharpoonup\xb$ in $L^p(0,\tau;\ss)$. Then, $\F(\xb_n)\rightharpoonup\F(\xb)$ in $L^p(0,\tau;\ss)$.}
\item \label{as:B2} Let $K_{ad}$ be a compact {convex} set in the actuator design space $\mathbb{K}$. The family of input operators $\mc{B}(\cdot):K_{ad}\to \mc{L}(\cs,\ss)$  are continuous with respect to $\rb$ in the operator norm topology:
\begin{equation}
\lim_{\rb_2\to \rb_1}\norm{\mc{B}(\rb_2)-\mc{B}(\rb_1)}{\mc{L}(\cs,\ss)} = 0. \notag
\end{equation}
\item The functionals $\phi(\cdot)$ and $\psi(\cdot)$ are weakly lower semi-continuous non-negative functionals on $\ss$ and $\cs$, respectively.\label{as:C3 B2}
\end{enumerate}
\end{assumption}
It is shown in \Cref{sec:existence of an optimizer} that under these assumptions, an optimal control and actuator design exist. 

\cref{thm:existence optimizer}: \textit{For  initial condition $\xb_0\in \ss$, let $\tau$ be such that the mild solution exists for all $\ub\in U_{ad}$, and all $\rb\in K_{ad}$. Under assumptions \ref{A} and \ref{B},  there exists a control input $\ub^o\in U_{ad}$ together with an actuator design $\rb^o\in K_{ad}$, that solves the optimization problem \ref{eq:optimal problem}.}%

To characterize an optimizer to the optimization problem, further assumptions on differentiability of the nonlinear operators $\mc F(\cdot)$ and $\B(\cdot)$, and the cost function are needed.
\begin{assumption}\leavevmode\label{C}
\begin{enumerate}
\item \label{as:C1} The nonlinear operator $\F(\cdot)$ is {G\^ateaux} differentiable on $\ss$ {(\cite[Def.~1.29]{hinze2008optimization})}. Indicate the {G\^ateaux} derivative of $\F(\cdot)$ at $\xb$ {in the direction $\pb$} by $\F'_{\xb}\pb$. {Furthermore, the mapping $\xb\mapsto \F'_{\xb}$ is bounded; that is, bounded sets in $\ss$ are mapped to bounded sets in $\mc L(\ss)$.}
\item \label{as:C2} The control operator $\mc{B}(\rb)$ is {G\^ateaux} differentiable with respect to $\rb$ {from $K_{ad}$ to $\mc{L}(\cs,\ss)$}. Indicate the {G\^ateaux} derivative of $\mc{B}(\rb)$ at $\rb^o$ {in the direction $\rb$} by $\mc{B}'_{\rb^o}\rb	$. {Furthermore, the mapping $\rb^o\mapsto \B'_{\rb^o}$ is bounded; that is, bounded sets in $\as$ are mapped to bounded sets in $\mc L(\as,\mc L(\cs,\ss))$.}
\item \label{as:C3} The spaces $\ss$, $\cs$, and $\as $ are Hilbert spaces, and {p=2}. Also, in the cost function, set
\begin{equation}\label{Q&R}
\phi(\xb)=\inn{\mc{Q}\xb}{\xb}, \quad \psi(\ub)=\inn{\mc{R}\ub}{\ub}_{\cs},
\end{equation}
where the linear operator $\mc{Q}$ is  a positive semi-definite, self-adjoint bounded operator on $\ss$, and the linear operator $\mc{R}$ is a positive definite, self-adjoint  bounded operator on $\cs$.
\end{enumerate}
\end{assumption}
 
Since $\ss$, $\cs $, and $\as$ are assumed to be Hilbert spaces in assumption \ref{C}\ref{as:C3}, the  dual of each of these spaces can be identified with the space itself. The operator $(\mc{B}'_{\rb^o}\ub)^*:\ss\to \mathbb{K}$ is defined as
\begin{equation}
\inn{(\mc{B}'_{\rb^o}\ub)^*\pb}{\rb}_\mathbb{K}=\inn{\pb}{(\mc{B}'_{\rb^o}\rb)\ub}, \quad \forall (\ub,\pb,\rb)\in  \cs\times\ss\times \mathbb{K}.\notag
\end{equation}
The following theorem is proved in \Cref{sec:characterizing}. In this theorem  $\xb=\mc{S}(\ub;\rb,\xb_0)$ denotes the control-to-state map  (see \cref{def:S}).

\cref{thm:characterizing}: \textit{Suppose assumptions \ref{A}\ref{as:A1} and \ref{C} hold, 
For any initial condition $\xb_0\in\ss$, let the pair $(\ub^o,\rb^o){\in U_{ad}\times K_{ad}}$ be a local minimizer of the optimization problem \ref{eq:optimal problem} with the optimal trajectory $\xb^o=\mc{S}(\ub^o;\rb^o,\xb_0)$ and let 
 $\pb^o(t)$, the adjoint state, indicate the mild solution of the final value problem
\begin{equation}
\dot{\pb}^o(t)=-(\mc{A}^*+\F'^*_{\xb^o(t)})\pb^o(t)-\mc{Q}\xb^o(t), \quad \pb^o(\tau)=0.\notag
\end{equation} 
If $(\ub^o,\rb^o)$ is in the interior of $ U_{ad}\times K_{ad}$ then $(\ub^o,\rb^o)$ satisfies
\begin{subequations}
\begin{flalign}
&\ub^o(t)=-\mc{R}^{-1}\mc{B}^*(\rb^o)\pb^o(t),\notag\\
&\int_0^{\tau} (\mc{B}'_{\rb^o}\ub^o(t))^*\pb^o(t)\, dt.\notag
\end{flalign}%
\end{subequations}}%

\section{Existence of a Solution to the IVP}\label{sec:existence of a solution}
In the existing literature, the existence of a unique local solution to (\ref{IVP}) is guaranteed for continuously differentiable control inputs (see e.g. \cite[Thm.~6.1.5]{pazy}). Requiring that  $\ub\in C^1(0,\tau;\cs)$ is too restrictive for establishing existence of an optimal control. The following theorem guarantees the existence of a unique local mild solution under  a weaker assumption on the input.
\begin{theorem} \label{thm:mild}
Under assumption \ref{A}, for each $\xb_0\in \ss$  and positive number $R$,  there exists $\tau>0$ such that (\ref{IVP}) admits a unique local mild solution $\xb\in C(0,\tau;\ss)$ for all {$\ub\in U_{ad}$}, and all $\rb\in K_{ad}$.
\end{theorem}
\begin{proof}
The idea of the proof is similar to  \cite[Thm.~6.1.4]{pazy}, with a slight modification that here $\ub(t)$ is in $L^p(0,\tau;\cs)$.
For any $\xb_0 \in \ss$ choose constants $\delta_0>0$ and $\tau>0$ such that for all $t\in[0,\tau]$
$$\normm{\mc{T}(t)\xb_0-\xb_0}\le\delta_0.$$
 Let $\mathbb{S}$ be the closed bounded subset of  $C(0,\tau;\ss)$ defined as
\begin{equation}
\mathbb{S}=\left\{ \xb\in C(0,\tau;\ss) | \; \xb(0)=\xb_0, \, \normm{\xb(t)-\xb_0}\le 2\delta_0, \; \forall t\in [0,\tau] \right\}.
\end{equation}
Define the operator $\mc{G}$ on $\mathbb{S}$ to be
\begin{equation}
\mc{G}(\xb)(t)=\mc{T}(t)\xb_0+\int_0^t \mc{T}(t-s)\F(\xb(s))\, ds +\int_0^t \mc{T}(t-s)\mc{B}(\rb)\ub(s)\,ds . 
\end{equation}
It will be shown that for sufficiently small  $\tau$, $\mc{G}$ maps $\mathbb{S}$ into $\mathbb{S}$ and is a contraction on $\mathbb S .$

Use the  triangle inequality and write
\begin{flalign}
\normm{\mc{G}(\xb)(t)-\xb_0}\le& \normm{\mc{T}(t)\xb_0-\xb_0}+\normm{\int_0^t \mc{T}(t-s)\F(\xb(s))\, ds}  \label{eq:d1}
\\
&+\normm{\int_0^t \mc{T}(t-s)\mc{B}(\rb)\ub(s)\,ds}  .\notag 
\end{flalign}
There exist a number $M_{\T}>0$ that $\normm{\mc{T}(t)}\le M_{\T}$ for all $t\in [0,\tau]$. Also, recall from assumption \ref{A}\ref{as:A2} that there is $L_{\F\delta}>0$ so that 
$\|\F(\xb(s))\|\le L_{\F\delta} \|\xb(s)\|$ on a ball of radius $\delta=\normm{\xb_0}+2\delta_0$ centered at the origin. 
This gives a bound for the second term on the left hand side of the inequality (\ref{eq:d1})
\begin{equation}
\normm{\int_0^t \mc{T}(t-s)\F(\xb(s))\, ds} \leq M_{\T}L_{\F\delta}\delta \tau . \label{eq-bd2}
\end{equation}
Using assumption \ref{A}\ref{as:A3}, an upper bound for the third right hand side term is
\begin{equation}
\normm{\int_{0}^{t} \mc{T}(t-s)\mc{B}(\rb)\ub(s)\,ds}\le  M_{\T}M_\mc{B}\norm{\ub}{p}\tau^{(p-1)/p}.  \label{eq-hold}
\end{equation}
Applying  these bounds to inequality (\ref{eq:d1}), it follows for all $\ub\in U_{ad}$ that 
\begin{equation}
\normm{\mc{G}(\xb)(t)-\xb_0}\le \delta_0+M_{\T}L_{\F\delta}\delta \tau+M_{\T}M_\mc{B}R \tau^{(p-1)/p}.
\label{eq-bd}
\end{equation}
Choose $\tau$ small enough that the right hand side in (\ref{eq-bd}) is less than $2 \delta_0 .$
For such $\tau$, $\mc G : \mathbb S \to \mathbb S .$

Because of the local Lipschitz continuity of $\F(\cdot)$
\begin{flalign}
\normm{\mc{G}(\xb_2)-\mc{G}(\xb_1)}_{C(0,\tau;\ss)}&\le \sup_{t\in[0,\tau]}\normm{\int_0^t \mc{T}(t-s)\left(\F(\xb_2(s))-\F(\xb_1(s))\right)\, ds}\notag\\
&\le M_{\T}L_{\F\delta}\tau\normm{\xb_2-\xb_1}_{C(0,\tau;\ss)} .
\end{flalign}
Choosing $\tau$ so $M_{\T}L_{\F\delta}\tau< 1$ yields that $\mc G$ is a contraction on $\mathbb S.$ Thus, the operator $\mc{G}$ has a unique fixed point in $\mathbb{S}$ that satisfies
\begin{equation}
\xb(t)=\mc{T}(t)\xb_0+\int_0^t \mc{T}(t-s)\F(\xb(s))\, ds +\int_0^t \mc{T}(t-s)\mc{B}(\rb)\ub(s)\,ds \, .  \label{eq:d2}
\end{equation}
Therefore, $\xb (t)$ is the  unique local mild solution of (\ref{IVP}).
\end{proof}

\begin{corollary}\label{corollary1}
{Under assumption \ref{A}, for all $\ub\in U_{ad}$,} there exists a positive number $c_{\tau}$ such that the mild solution to (\ref{IVP}) satisfies
\begin{equation}
\norm{\xb}{{C(0,\tau;\ss)}}\le c_{\tau} \left(\normm{\xb_0}+\norm{\mc{B}(\rb)}{\mc{L}(\cs,\ss)}\norm{\ub}{p}\right) .   \label{eq:estimate on mild solution}
\end{equation}

\begin{proof}
{Let $\tau$ be as in \cref{thm:mild}}. Take the norm of both sides of (\ref{eq:mild solution}) and apply assumption \ref{A} together with the triangle inequality to obtain
\begin{flalign}
\normm{\xb(t)}&\le \normm{\mc{T}(t)\xb_0}+\normm{\int_0^t \mc{T}(t-s)\F(\xb(s))\, ds}+\normm{\int_0^t \mc{T}(t-s)\mc{B}(\rb)\ub(s)\,ds}\notag\\
&\le M_{\T}\normm{\xb_0}+M_{\T}L_{\F\delta}\int_0^t \normm{\xb(t)}dt\label{ineq}\\
&\quad +M_{\T}\tau^{(p-1)/p}\norm{\mc{B}(\rb)}{\mc{L}(\cs,\ss)}\norm{\ub}{p}.\notag
\end{flalign}
Defining the constant {$$c_{\tau}=\max\left\{1,M_{\T}e^{M_{\T}L_{\F\delta}\tau}\right\},$$}%
and applying Gronwall's lemma \cite[Thm.~1.4.1]{zettl2005} to {inequality (\ref{ineq})} yield
\begin{equation}
\normm{\xb(t)}\le c_{\tau}\left(\normm{\xb_0}+\norm{\mc{B}(\rb)}{\mc{L}(\cs,\ss)}\norm{\ub}{p}\right).
\end{equation}
{Taking supremum of both side over $[0,\tau]$ results in (\ref{eq:estimate on mild solution}).} 
\end{proof}
\end{corollary}

\section{Existence of an Optimizer}\label{sec:existence of an optimizer}

The following theorem ensures that the optimization problem \ref{eq:optimal problem} admits an optimal control input $\ub^o\in U_{ad}$ together with an optimal actuator design $\rb^o\in K_{ad}$.

\begin{theorem} \label{thm:existence optimizer}
For  initial condition $\xb_0\in \ss$, let $\tau$ be such that the mild solution exists for all $\ub\in U_{ad}$, and all $\rb\in K_{ad}$. Under assumptions \ref{A} and \ref{B},  there exists a control input $\ub^o\in U_{ad}$ together with an actuator design $\rb^o\in K_{ad}$, that solves the optimization problem \ref{eq:optimal problem}.
\end{theorem}
\begin{proof}
The cost function $J(\ub,\rb;\xb_0)$ is bounded from below, and thus it has an infimum, say $j(\xb_0)$. This infimum is finite by assumption. As a result, there is a sequence of inputs $\ub_n\in U_{ad}$ and actuator design $\rb_n\in K_{ad}$ such that
\begin{equation}
\lim_{n \to \infty} J(\ub_n,\rb_n;\xb_0)= j(\xb_0) .
\end{equation}

The set $U_{ad}$ is a  bounded subset of the reflexive space $L^p(0,\tau;\cs)$, $1<p<\infty$, and hence it is weakly sequentially  compact {\cite[Thm.~9.4.3]{wouk1979course}}. Since $U_{ad}$ is closed and convex, it is also weakly closed \cite[Thm.~2.11.]{troltzsch2010optimal}. These statements mean that there is a subsequence of $\ub_n$ that converges weakly to some element $\ub^o\in U_{ad}$. To simplify the notation, we denote the weakly convergent subsequence by $\ub_n$:
\begin{equation}
\ub_n(t)\rightharpoonup  \ub^o(t) \quad \text{as} \quad n\to \infty.
\end{equation} 
The compactness of $K_{ad}$ implies that there is also a subsequence of $\rb_n$ that converges to some  $\rb^o$ in $K_{ad}$. This subsequence is also indicated by $\rb_n$; that is
\begin{equation}
\rb_n\to \rb^o \quad \text{as} \quad n\to \infty.
\end{equation}
Using assumption \ref{B}\ref{as:B2}, it follows that
\begin{equation}
\mc{B}(\rb_n)\ub_n(t)\rightharpoonup \mc{B}(\rb^o)\ub^o(t) \quad \text{in} \quad L^p(0,\tau;\ss) \quad \text{as} \quad n\to \infty \label{eq:d8}.
\end{equation}
{According to Proposition 1.84 of \cite{barbu2012convexity}, every continuous linear map is weakly continuous, yielding
\begin{equation}\label{eq1}
\int_0^t \T(t-s)\B(\rb_n)\ub_n(s)ds\rightharpoonup \int_0^t \T(t-s)\mc{B}(\rb^o)\ub^o(s)ds \quad \text{in } C(0,\tau;\ss).
\end{equation}}

Moreover, by \cref{thm:mild}, for every pair $(\ub_n,\rb_n)$, there exists a state $\xb_n(t) \in C(0,\tau;\ss)$. The sequence $\{ \xb_n(t) \}$ is also bounded in $C(0,\tau;\ss)$ by \cref{corollary1}; that is
\begin{equation}
\|\xb_n\|_{C(0,\tau;\ss)} \le c_{\tau}\left(\normm{\xb_0}+M_\mc{B}R\right).
\end{equation}
The sequence $\xb_n(t)$ is bounded in $C(0,\tau;\ss)$ and so in $L^p(0,\tau;\ss)$ as well. The latter is a reflexive Banach space; this means that a subsequence of $\xb_n(t)$, denote it by $\xb_n(t)$ for simplicity, weakly converges to an element of $\xb^o$ in $L^p(0,\tau;\ss)$. {By assumption \ref{B}\ref{as:B1}, it follows that
\begin{equation}\label{eq:d9}
\F(\xb_n(t)) \rightharpoonup \F(\xb^o(t)), \quad \text{in } L^p(0,\tau;\ss),
\end{equation}
and also by Proposition 1.84 of \cite{barbu2012convexity}
\begin{equation}\label{eq2}
\int_0^t\T(t-s)\F(\xb_n(s))ds \rightharpoonup \int_0^t\T(t-s)\F(\xb^o(s))ds, \quad \text{in} \quad C(0,\tau;\ss).
\end{equation}
Recall that each $(\xb_n,\ub_n, \rb_n)$  satisfies
\begin{equation}\label{eq3}
\xb_n(t)=\mc{T}(t)\xb_0+\int_0^t\mc{T}(t-s)\F(\xb_n(s))ds+\int_0^t\T(t-s)\mc{B}(\rb_n)\ub_n(s) ds .
\end{equation}
Apply (\ref{eq1}) and (\ref{eq2}) to (\ref{eq3}), it follows that $\xb^o(t)$ is in $C(0,\tau;\ss)$. Note that the mild solution is unique; thus, $\xb^o(t)$ is the mild solution to IVP (\ref{IVP}) with input $\ub^o(t)$ and actuator design $\rb^o$, satisfying
\begin{equation}
\xb^o(t)=\mc{T}(t)\xb_0+\int_0^t\mc{T}(t-s)\F(\xb^o(s))ds+\int_0^t\T(t-s)\mc{B}(\rb^o)\ub^o(s)ds.
\end{equation}}

It  remains to show that $(\xb^o(t),\ub^o(t),\rb^o)$ minimizes $J(\ub,\rb;\xb_0)$. 
Recall from definition of the sequence $\ub_n$ and $\rb_n$ that
\begin{flalign}
j(\xb_0)&=\liminf_{n\to \infty} J(\ub_n,\rb_n;\xb_0)\notag \\
&=\liminf_{n\to \infty}\int_0^{\tau} \phi(\xb_n(t)) \, dt+\liminf_{n\to \infty} \int_0^{\tau} \psi(\ub_n(t)) \, dt.
\end{flalign}
 From assumption \ref{B}\ref{as:C3 B2}, the cost function is weakly lower semi-continuous in $\xb$ and $\ub$. This {together with Fatou's Lemma} implies
\begin{flalign}
j(\xb_0)\ge \int_0^{\tau} \phi(\xb^o(t)) \, dt+\int_0^{\tau} \psi(\ub^o(t)) \, dt=J(\ub^o,\rb^o;\xb_0) .
\end{flalign} 
Since $j(\xb_0)$ was defined to be the infimum, 
$$ j(\xb_0) =J(\ub^o,\rb^o;\xb_0). $$

 Therefore, for every initial condition $\xb_0\in \ss$, there exists an  control input $\ub^o(t)$ together with an actuator design $\rb^o$, with corresponding mild solution  $\xb^o(t)$ that achieves the minimum value of the cost function. 
\end{proof}

For a linear partial differential equation and quadratic cost, the optimal actuator problem may not be convex; see for example \cite[Fig. 7]{morris2011linear}. Uniqueness of the optimal control and actuator is not guaranteed.

\section{Optimality Conditions}\label{sec:characterizing}

In order to establish the first order optimality condition for an optimizer $(\ub^o,\rb^o)$, further regularity on the control-to-state map is needed.

\begin{definition}\label{def:S}
For each initial condition $\xb_0\in \ss$, and actuator design  $\rb\in K_{ad}$, the control-to-state operator is the operator $\mc{S}(\ub;\rb,\xb_0):U_{ad}\subset (L^p(0,\tau;\cs))\to L^p(0,\tau;\ss)$  that maps every input $\ub\in U_{ad}$ to the state $\xb\in L^p(0,\tau;\ss).$ It is described by
\begin{equation}
\xb(t)=\mc{T}(t)\xb_0+\int_0^t \mc{T}(t-s)\F(\xb(s))ds+\int_0^t \mc{T}(t-s)\mc{B}(\rb)\ub(s)ds.\notag
\end{equation}
\end{definition}

In next two theorems, it is proved that under certain assumptions, the control-to-state map is Lipschitz continuous in both $\ub$ and $\rb$. For the Lipschitz continuity with respect to the actuator design, a stronger assumption on the input operator $\mc{B}(\rb)$ than continuity in $\rb$ is needed. 

\begin{proposition}\label{pro:lipschitz S}
\begin{enumerate}[label=(\alph*),leftmargin=0cm,itemindent=1cm,labelwidth=\itemindent,labelsep=0cm,align=left]
\item \label{-a} Under assumption \ref{A}, for any initial condition $\xb_0\in \ss$, the control-to-state map $\mc{S}(\ub;\rb,\xb_0)$ is Lipschitz continuous in $\ub$, i.e., there exists a positive constant $L_{\ub}$ such that 
\begin{equation}\label{eq:lip in u}
\norm{\mc{S}(\ub_2;\rb,\xb_0)-\mc{S}(\ub_1;\rb,\xb_0)}{C(0,\tau;\ss)}\le L_{\ub} \norm{\ub_2-\ub_1}{p},
\end{equation}
for all $\ub_1$ and $\ub_2$ in $U_{ad}$, and $\rb\in K_{ad}$.

\item \label{-b} Under extra assumptions \ref{C}\ref{as:C2} and {the space $\as$ being a Banach space}, the control-to-state map $\mc{S}(\ub;\rb,\xb_0)$ is Lipschitz continuous in $\rb$, i.e., there exists a positive constant $L_{\rb}$ such that
\begin{equation}\label{eq:lip in r}
\norm{\mc{S}(\ub;\rb_2,\xb_0)-\mc{S}(\ub;\rb_1,\xb_0)}{C(0,\tau;\ss)}\le L_{\rb} \norm{\rb_2-\rb_1}{\mathbb{K}},
\end{equation}
for all $\rb_1$ and $\rb_2$ in $K_{ad}$, and $\ub\in U_{ad}$.
\end{enumerate}
\end{proposition}
The proof of this proposition is straightforward; a proof is provided in \cref{ap:appendix a}.

{G\^ateaux} differentiability of the control-to-state map as well as its derivatives need to be formulated in order to characterize an optimizer.

For any $\xb^o \in C(0,\tau;\ss)$ define the time-varying  operator operator $\F'_{\xb^o(t)}$. At any $t>0$, this operator is linear on $\ss$. 
Consider the time-varying IVP  
\begin{equation}
\dot{\tilde{\xb}}(t)=(\mc{A}+\F'_{\xb^o(t)})\tilde{\xb}(t)+\mc{B}(\rb)\tilde{\ub}(t), \quad \tilde{\xb}(0)=0. \label{eq:ztilde_dum}
\end{equation}
The mild solution  is described by a two-parameter family of operators, say $\mc{U}(t,s)$, known as an evolution operator. 
 
 The following lemma relies on the existence results: Theorem 5.5.6 and Theorem 5.5.10 in \cite{fattorini1999infinite}.

\begin{lemma}
\begin{enumerate}[label=(\alph*),leftmargin=0cm,itemindent=1cm,labelwidth=\itemindent,labelsep=0cm,align=left]\label{lem:adjoint}
\item \label{a} The mild solution of IVP problem (\ref{eq:ztilde_dum}) is described by
\begin{equation}
\tilde{\xb}(t)=\int_0^t \mc{U}(t,s)\mc{B}(\rb)\tilde{\ub}(s) \, ds, \label{evolution-solution}
\end{equation}
in which $\mc{U}(t,s)$ is a strongly continuous evolution operator on $\ss$ for $0\le s\le t\le \tau$. 
\item \label{b} Let $\fb\in L^1(0,\tau;\ss)$, and consider the following final value problem (FVP) backward in time
\begin{equation}
\dot{\tilde{\pb}}(s)=-(\mc{A}^*+\F'^*_{\xb^o(s)})\tilde{\pb}(s)-\fb(s), \quad \tilde{\pb}(\tau)=0, \label{eq:FVP}
\end{equation}
then, the mild solution of this evolution equation satisfies
\begin{equation}
\tilde{\pb}(s)=\int_s^{\tau} \mc{U}^*(t,s)\fb(t) \, dt,
\end{equation}
where $\mc{U}^*(t,s)$ is the adjoint of $\mc{U}(t,s)$ on $\ss$ for every $0\le s\le t \le \tau$.
\end{enumerate}
\end{lemma}
\begin{proof}
The time-invariant part of the state operator in (\ref{eq:ztilde_dum}),  $\mc{A}$, is the  generator of an strongly continuous semigroup. 
According to  \cite[Thm. 5.5.6]{fattorini1999infinite}, in  order for a  strongly continuous evolution operator $\mc{U}(t,s)$ to  exist so that (\ref{evolution-solution}) is the mild solution to the (\ref{eq:ztilde_dum}), it is sufficient that
for every $\tilde{\xb}\in \ss$ the mapping $t\mapsto\F'_{\xb^o(t)}\tilde{\xb}$ is strongly measurable and that a function $\alpha(t)\in L^1(0,\tau)$ exists such that
\begin{equation}
\normm{\F'_{\xb^o(t)}}\le \alpha(t), \quad t\in [0,\tau]. \label{eq:alpha}
\end{equation}
By assumption \ref{C}\ref{as:C1}, since the state $\xb^o(t)$ is uniformly bounded, the operator norm of $\F'_{\xb^o(t)}$ admits an upper bound for all $t\in[0,\tau]$. Consequently, a strongly continuous evolution operator $\mc{U}(t,s)$ exists so that (\ref{evolution-solution}) is the mild solution to (\ref{eq:ztilde_dum}).

Since the state space  $\ss$ is a separable reflexive Banach space, Theorem 5.5.10 of \cite{fattorini1999infinite} implies that the mild solution of (\ref{eq:FVP}) is described by an evolution operator. Moreover,  for every $0\le s\le t \le \tau$, this evolution operator is the adjoint on $\ss$ of the evolution operator $\mc{U}(t,s) .$
\end{proof}

\begin{proposition}\label{pro:frechet derivative}
Under assumption \ref{A}, and \ref{C}\ref{as:C1}, for every initial condition $\xb_0\in \ss$ and actuator design $\rb\in K_{ad}$, the control-to-state map $\mc{S}(\ub;\rb,\xb_0)$ is {G\^ateaux} differentiable in $\ub$ in the interior of $U_{ad}$. The {G\^ateaux} derivative of $\mc{S}(\ub;\rb,\xb_0)$ at $\ub^o$ {in the direction $\tilde{\ub}$} is
\begin{equation}
\mc{S}'_{\ub^o}\tilde{\ub}=\tilde{\xb}, \quad \forall \tilde{\ub}\in L^p(0,\tau;\mathbb{U}), \label{eq:derivative 1}
\end{equation}
where, defining $\xb^o(t)=\mc{S}(\ub^o;\rb,\xb_0)$,  $\tilde{\xb}$ is the mild solution to the IVP
\begin{equation}
\dot{\tilde{\xb}}(t)=(\mc{A}+\F'_{\xb^o(t)})\tilde{\xb}(t)+\mc{B}(\rb)\tilde{\ub}(t), \quad \tilde{\xb}(0)=0 \label{eq:ztilde}.
\end{equation}
The mild solution to this equation is given by the evolution operator  $\, \mc{U}(t,s)$  in \Cref{lem:adjoint}\ref{a}.
\end{proposition}
\begin{proof}
For sufficiently small $\epsilon$, there is a mild solution to IVP (\ref{IVP}) with input $\ub^o+\epsilon \tilde{\ub}$. Denote by $\xb=\mc{S}(\ub^o+\epsilon \tilde{\ub};\rb,\xb_0)$ the mild solution to the IVP
\begin{equation}
\dot{\xb}(t)=\mc{A}\xb(t)+\F(\xb(t))+\mc{B}(\rb)(\ub^o(t)+\epsilon \tilde{\ub}(t)), \quad \xb(0)=\xb_0.\label{eq:zp}
\end{equation}
The state $\xb^o=\mc{S}(\ub^o;\rb,\xb_0)$ is by definition the mild solution of the IVP 
\begin{equation}
\dot{\xb}^o(t)=\mc{A}\xb^o(t)+\F(\xb^o(t))+\mc{B}(\rb)\ub^o(t)  , \quad \xb^o(0)=\xb_0.    \label{eq:zstar}
\end{equation}
Define $\xb_{e}=(\xb-\xb^o)/\epsilon-\tilde{\xb}$, subtract the equations (\ref{eq:zstar}) and (\ref{eq:ztilde}) from (\ref{eq:zp}) to obtain
\begin{flalign}\label{eq:c1}
\dot{\xb}_{e}(t)=&(\mc{A}+\F'_{\xb^o(t)})\xb_{e}(t)\notag\\
&+\frac{1}{\epsilon}\left(\F(\xb(t))-\F(\xb^o(t))-\F'_{\xb^o(t)}(\xb(t)-\xb^o(t))\right), \quad \xb_{e}(0)=0.
\end{flalign}
Define $\eb_\F(t)$ as
\begin{equation}\label{e(t)}
\eb_\F(t)\coloneqq\frac{1}{\epsilon}\left(\F(\xb(t))-\F(\xb^o(t))-\F'_{\xb^o(t)}(\xb(t)-\xb^o(t))\right)
\end{equation}
Assumption \ref{C}\ref{as:C1} ensures that for each $t\in [0,\tau]$,  $\eb_\F(t)\to 0$ as $\epsilon \to 0$. It will be shown that $\eb_\F(t)$ is uniformly bounded. By \Cref{corollary1}, the norm of the states $\xb(t)$ and $\xb^o(t)$ is uniformly bounded over $[0,\tau]$ by some number $\delta$,
\begin{equation}
\delta\le c_\tau\left(\normm{\xb_0}+M_\B R\right).
\end{equation}
Use the local Lipschitz continuity of $\F(\cdot)$ (assumption \ref{A}\ref{as:A2}) and \Cref{pro:lipschitz S}\ref{-a} to obtain
\begin{flalign}
\frac{1}{\epsilon}\normm{\F(\xb(t))-\F(\xb^o(t))}&\le \frac{1}{\epsilon}L_{\F\delta}\normm{\xb(t)-\xb^o(t)}\notag\\
&\le L_{\F\delta} L_\ub \norm{\tilde{\ub}}{p}.\label{eq8}
\end{flalign}
Letting $M_{\F'}=\sup \{\|\F'_{\xb^o(t)}\|:\; t\in[0,\tau]\}$, assumption \ref{C}\ref{as:C1} together with \Cref{pro:lipschitz S}\ref{-a} also yields
\begin{flalign}\label{eq9}
\frac{1}{\epsilon}\normm{\F'_{\xb^o(t)}(\xb(t)-\xb^o(t))}
&\le M_{\F'}L_\ub\norm{\tilde{\ub}}{p}.
\end{flalign} 
Combining (\ref{eq8}) and (\ref{eq9}) leads to 
\begin{equation}\label{bound on e(t)}
\normm{\eb_\F(t)}\le \left(L_{\F\delta}+M_{\F'}\right)L_\ub\norm{\tilde{\ub}}{p}, \quad \forall t\in [0,\tau].
\end{equation}

Now substitute (\ref{e(t)}) into (\ref{eq:c1}).  The state $\xb_{e}$ is the mild solution to  the IVP
\begin{equation}
\dot{\xb}_{e}(t)=(\mc{A}+\F'_{\xb^o(t)})\xb_{e}(t)+\eb_\F(t), \quad \xb_{e}(0)=0. \label{eq:lastone1} 
\end{equation}
Recall that the mild solution of this evolution equation is described by an evolution operator $\mc{U}(t,s)$ by \Cref{lem:adjoint}\ref{a}. Let $M_{\mc{U}}$ be an upper bound for the operator norm of $\mc{U}(t,s)$ over $0\le t\le s \le \tau$, the mild solution to (\ref{eq:lastone1}) satisfies the estimate
\begin{flalign}
\norm{\xb_{e}}{L^p(0,\tau;\ss)}&\le \tau^{1/p}\norm{\xb_{e}}{C(0,\tau;\ss)} \nonumber \\
&\le \tau^{1/p}M_{\mc{U}}\int_0^{\tau}\normm{\eb_\F(t)}dt. \label{eq-int}
\end{flalign}
Since $\lim_{\epsilon  \to 0 } \normm{\eb_\F(t)}= 0$ for each $t \in [0,\tau]$ and $\normm{\eb_\F(t)}$ is uniformly bounded over $[0,\tau]$ for all $\epsilon ,$ the  bounded convergence theorem implies that  the integral in (\ref{eq-int}) converges to zero. Thus,
\begin{equation}\label{eq:c6}
\lim_{\epsilon\to 0}\norm{\frac{1}{\epsilon}\left(\mc{S}(\ub^o+\epsilon\tilde{\ub};\rb,\xb_0)-\mc{S}(\ub^o;\rb,\xb_0)\right)-\mc{S}'_{\ub^o}{\tilde{\ub}}}{L^p(0,\tau;\ss)}=\lim_{\epsilon\to 0}\norm{\xb_{e}}{L^p(0,\tau;\ss)}=0.\notag
\end{equation}
This proves that $\mc{S}'_{\ub^o}{\tilde{\ub}}$ is the G\^ateaux derivative of $\mc{S}(\ub;\rb,\xb_0)$ at $\ub^o$ in the direction $\tilde{\ub}$.
\end{proof}

\begin{proposition}\label{pro:frechet derivative in r}
Under assumption \ref{A}, \ref{C}\ref{as:C1}-\ref{C}\ref{as:C3}, for every initial condition $\xb_0\in \ss$ and control input $\ub\in U_{ad}$, the control-to-state map $\mc{S}(\ub;\rb,\xb_0)$ is {G\^ateaux} differentiable in $\rb$ in the interior of $K_{ad}$. The {G\^ateaux} derivative of $\mc{S}(\ub;\rb,\xb_0)$ at $\rb^o$ {in the direction $\tilde{\rb}$} is
\begin{equation}
\mc{S}'_{\rb^o}\tilde{\rb}=\tilde{\yb}, \quad \forall \tilde{\rb}\in \mathbb{K},
\end{equation}
where, defining $\xb^o(t)=\mc{S}(\ub;\rb^o,\xb_0)$,  $\tilde{\yb}$ is the mild solution to the IVP
\begin{equation}
\dot{\tilde{\yb}}(t)=(\mc{A}+\F'_{\xb^o(t)})\tilde{\yb}(t)+\left(\mc{B}'_{\rb^o}\tilde{\rb}\right)\ub(t), \quad \tilde{\yb}(0)=0. \label{eq:ytilde}
\end{equation}
\end{proposition}
The proof of this proposition is similar to that of \cref{pro:frechet derivative}; a proof is provided in \cref{ap:appendix d}. 

Now that  differentiability and derivatives of the control-to-state map has been established, the first order necessary conditions for a pair $(\ub^o,\rb^o)$ to be a local optimizer can be derived. In order to place the problem in a Hilbert space, assumptions \ref{C}\ref{as:C3} and \ref{C}\ref{as:C3} are used, assuming that the spaces are Hilbert spaces and defining a cost function. It will also be assumed that $p=2$, considering inputs in $L^2(0,\tau;\cs)$. It is shown in the following lemma that this cost function is consistent with previous assumptions on the cost function (assumption \ref{B}\ref{as:C3 B2}).

\begin{lemma}\label{lem:weak lower semi}
The cost function in assumption \ref{C}\ref{as:C3} satisfies assumption \ref{B}\ref{as:C3 B2}; that is, it is weakly lower semi-continuous in $\xb$ and $\ub$.
\end{lemma} 
\begin{proof}
The cost function $J(\ub,\rb;\xb_0)$ in assumption \ref{C}\ref{as:C3} is continuous and convex function in both $\xb$ and $\ub$. That is, {letting $\lambda\in (0,1)$},
\begin{gather}
\int_0^{\tau}\inn{\mc{Q}\xb_n(t)}{\xb_n(t)}\, dt\to \int_0^{\tau}\inn{\mc{Q}\xb(t)}{\xb(t)}\, dt \quad \text{as} \quad \xb_n\to \xb \quad \text{in} \quad L^p(0,\tau;\ss),\notag\\ 
\inn{\lambda \mc{Q}\xb_1 +(1-\lambda)\mc{Q}\xb_2}{\lambda\xb_1+(1-\lambda)\xb_2}\le \lambda \inn{\mc{Q}\xb_1}{\xb_1}+(1-\lambda)\inn{\mc{Q}\xb_2}{\xb_2},\notag
\end{gather}
and a similar argument for $\ub$. {According to Theorem 13.2.2 in \cite{wouk1979course} and the corollary thereafter}, if a functional defined on a Banach space is continuous and convex; then, it is also weakly lower semi-continuous. Therefore, the cost function $J(\ub,\rb;\xb_0)$ is weakly lower semi-continuous in both $\xb$ and $\ub$.
\end{proof}
The next theorem derives the first order necessary conditions for an optimizer of the optimization problem \ref{eq:optimal problem}.

\begin{theorem}\label{thm:characterizing}
Suppose assumptions \ref{A}\ref{as:A1} and \ref{C} hold,
For any initial condition $\xb_0\in\ss$, let the pair $(\ub^o,\rb^o){\in U_{ad}\times K_{ad}}$ be a local minimizer of the optimization problem \ref{eq:optimal problem} with the optimal trajectory $\xb^o=\mc{S}(\ub^o;\rb^o,\xb_0).$ Let 
 $\pb^o(t)$, the adjoint state, indicate the mild solution of the final value problem
\begin{equation}
\dot{\pb}^o(t)=-(\mc{A}^*+\F'^*_{\xb^o(t)})\pb^o(t)-\mc{Q}\xb^o(t), \quad \pb^o(\tau)=0.\notag
\end{equation}  
If $(\ub^o,\rb^o)$ is in the interior of $ U_{ad}\times K_{ad}$ then $(\ub^o,\rb^o)$ satisfies
\begin{subequations}\label{eq:optimality}
\begin{flalign}
&\ub^o(t)=-\mc{R}^{-1}\mc{B}^*(\rb^o)\pb^o(t),\label{eq:optimality1}\\
&\int_0^{\tau} (\mc{B}'_{\rb^o}\ub^o(t))^*\pb^o(t)\, dt.\label{eq:optimality2}
\end{flalign}%
\end{subequations}
\end{theorem}

\begin{proof}
{To derive the optimality conditions (\ref{eq:optimality}), the G\^ateaux derivative of the cost function $J(\ub,\rb;\xb_0)$ with respect to $\ub\in U_{ad}$ and $\rb\in K_{ad}$ is calculated.} Use assumption \ref{C}\ref{as:C3}, the cost function is sum of two inner products in the Hilbert spaces $L^2(0,\tau;\ss)$ and $L^2(0,\tau;\cs)$; that is
\begin{equation}
J(\ub,\rb;\xb_0)=\inn{\mc{Q}\xb}{\xb}_{L^2(0,\tau;\ss)}+\inn{\mc{R}\ub}{\ub}_{L^2(0,\tau;\cs )}.
\end{equation}
Thus, {G\^ateaux} derivative of $J$ at $\ub^o$ along $\hb_{\ub}$ is
\begin{align}
J'_{\ub^o}\hb_{\ub}&=2\inn{\mc{Q}\mc{S}(\ub^o;\rb^o,\xb_0)}{\mc{S}'_{\ub^o}\hb_{\ub}}_{L^2(0,\tau;\ss)}+2\inn{\mc{R}\ub^o}{\hb_{\ub}}_{L^2(0,\tau;\cs )} \nonumber \\
&=2\inn{\mc{S}_{\ub^o}'^*\mc{Q}\mc{S}(\ub^o;\rb^o,\xb_0)+\mc{R}\ub^o}{\hb_{\ub}}_{L^2(0,\tau;\cs )}. \label{eq:derivative cost}
\end{align}
To calculate the adjoint operator $\mc{S}_{\ub^o}'^*$, let $\tilde{\ub}(t) \in L^2 (0,\tau; \mathbb U)$, $\tilde{\xb}(t) \in L^2 (0,\tau; \ss)$  be arbitrary. Using \cref{lem:adjoint},
\begin{flalign}
\inn{\tilde{\xb}}{{\mc{S}'_{\ub^o}}\tilde{\ub}}_{L^2(0,\tau;\ss )} &=\int_0^{\tau}\inn{\tilde{\xb}(t)}{\int_0^t\mc{U}(t,s)\mc{B}(\rb^o)\tilde{\ub}(s)\,ds}\,dt\notag\\
&=\int_0^{\tau}\int_s^{\tau}\inn{\tilde{\xb}(t)}{\mc{U}(t,s)\mc{B}(\rb^o)\tilde{\ub}(s)}\,dtds\notag\\
&=\int_0^{\tau}\inn{\mc{B}^*(\rb^o)\int_s^{\tau}\mc{U}^*(t,s)\tilde{\xb}(t) dt}{\tilde{\ub}	(s)}_\cs\,ds.
\end{flalign}
Thus,
$$(\mc{S}_{\ub^o}'^*\tilde{\xb})(s)=\mc{B}^*(\rb^o)\int_s^{\tau}\mc{U}^*(t,s)\tilde{\xb}(t) dt. $$
Define  $\tilde{\pb}(s)=\int_s^{\tau}\mc{U}^*(t,s)\tilde{\xb}(t) dt. $ By \cref{lem:adjoint}\ref{b}, $\tilde{\pb}(s) $  is the mild solution of the following FVP solved backward in time
\begin{equation}
\dot{\tilde{\pb}}(s)=-(\mc{A}^*+\F'^*_{\xb^o(s)})\tilde{\pb}(s)-\tilde{\xb}(s), \quad \tilde{\pb}(\tau)=0 \label{eq:adjoint implict} .
\end{equation}
It follows that
\begin{equation}\label{eq:d14}
(\mc{S}_{\ub^o}'^*\tilde{\xb})(s)=\mc{B}^*(\rb^o)\tilde{\pb}(s).
\end{equation} 
Incorporating (\ref{eq:d14}) into (\ref{eq:derivative cost}), and using Theorem 1.46 of \cite{hinze2008optimization} yields the optimality condition
\begin{equation}
\inn{\mc{B}^*(\rb^o)\pb^o+\mc R\ub^o}{\ub-\ub^o}_{{L^2(0,\tau;\cs )}}\ge 0, \quad \forall \ub \in U_{ad},
\end{equation}
where $\pb^o(s)$ solves 
\begin{equation}
\dot{\pb}^o(s)=-(\mc{A}+\F'_{\xb^o(s)})^*\pb^o(s)-\mc{Q}\xb^o(s), \quad \pb^o(\tau)=0.
\label{p0}
\end{equation}
Since $\mc R$ is positive-definite, and hence, invertible, inequality (\ref{eq:optimality1}) follows.
 
Taking the directional derivative of $J(\ub^o,\cdot\, ; \xb_0)$ at $\rb^o$ along $\hb_{\rb}$ yields
\begin{flalign}
J'_{\rb^o}\hb_{\rb}&=2\inn{\mc{Q}\mc{S}(\ub^o;\rb^o,\xb_0)}{\mc{S}'_{\rb^o}\hb_{\rb}}_{L^2(0,\tau;\ss )} \notag \\
&=2\inn{\mc{S}_{\rb^o}'^*\mc{Q}\mc{S}(\ub^o;\rb^o,\xb_0)}{\hb_{\rb}}_{\mathbb{K}}.
\end{flalign}
To calculate the adjoint operator $\mc{S}_{\rb^o}'^*$, use \cref{lem:adjoint}\ref{b}, and proceed as follows 
\begin{flalign}\label{eq:d5}
\inn{\mc{Q}\mc{S}(\ub^o;\rb^o,\xb_0)}{\mc{S}'_{\rb^o}\hb_{\rb}}_{L^2(0,\tau;\ss )}&=\int_0^{\tau}\inn{\mc{Q}\xb^o(t)}{\int_0^t\mc{U}(t,s)(\mc{B}'_{\rb^o}\hb_{\rb})\ub^o(s)\, ds} dt\notag\\
&=\int_0^{\tau}\inn{\int_s^{\tau}\mc{U}^*(t,s)\mc{Q}\xb^o(t)\,dt}{(\mc{B}'_{\rb^o}\hb_{\rb})\ub^o(s)} ds\notag\\
&=\int_0^{\tau}\inn{\pb^o(s)}{(\mc{B}'_{\rb^o}\hb_{\rb})\ub^o(s)} ds.
\end{flalign}
For each $\ub\in \cs$,  $(\mc{B}'_{\rb^o}\hb_{\rb})u$ is an element of $\ss .$ This defines a bounded linear map from $\hb_{\rb}\in \mathbb K$ to $\ss .$
There exists a bounded linear operator $(\mc{B}'_{\rb^o}\ub)^*$: $\ss\to \mathbb{K}$ satisfying
\begin{equation}
\inn{(\mc{B}'_{\rb^o}\ub)^*\pb}{\hb_{\rb}}_\mathbb{K}=\inn{\pb}{(\mc{B}'_{\rb^o}\hb_{\rb})u}.
\end{equation}
Incorporate this into (\ref{eq:d5}) to obtain
\begin{align}
\inn{\mc{Q}\mc{S}(\ub^o;\rb^o,\xb_0)}{\mc{S}'_{\rb^o}\hb_{\rb}}_{L^2(0,\tau;\ss )}&=\int_0^{\tau} \inn{(\mc{B}'_{\rb^o}\ub^o(s))^*\pb^o(s)}{\hb_{\rb}}_\mathbb{K} ds\notag \\
&=\inn{\int_0^{\tau} (\mc{B}'_{\rb^o}\ub^o(s))^*\pb^o(s)\, ds}{\hb_{\rb}}_\mathbb{K}.\label{eq4}
\end{align}
This gives an explicit form of the adjoint operator $\mc{S}_{\rb^o}'^*$. Similarly, by Theorem 1.46 of \cite{hinze2008optimization}, inner product (\ref{eq3}) must be non-negative for any direction $\rb-\rb^o$ in $K_{ad}$ yielding (\ref{eq:optimality2}).
\end{proof}
{\begin{remark}Under assumptions \ref{A}\ref{as:A1} and \ref{C}, if there is an open set $U\times K\supset U_{ad}\times K_{ad}$ on which the control-to-state map $\mc S(\ub;\rb,\xb_0)$ is G\'ateaux differentiable in $(\ub,\rb)$, then, using \cite[Thm.146]{hinze2008optimization} and a proof identical to that of Theorem \ref{thm:characterizing}, it follows that  every local minimizer $(\ub^o,\rb^o)\in U_{ad}\times K_{ad}$ satisfies
\begin{subequations}\label{eq:optimality-ineq}
\begin{flalign}
\label{eq:optimality3}\inn{\ub^o+\mc{R}^{-1}\mc{B}^*(\rb^o)\pb^o}{\ub-\ub^o}_{L^2(0,\tau;\cs )}\ge 0, \quad &\forall \ub \in U_{ad},\\
\label{eq:optimality4}\inn{\int_0^{\tau} (\mc{B}'_{\rb^o}\ub^o(t))^*\pb^o(t)\, dt}{\rb-\rb^o}_{\as}\ge 0, \quad &\forall \rb\in K_{ad}.
\end{flalign}%
\end{subequations}
\end{remark}}

Together with the original PDE, \cref{thm:characterizing} provides the following system of equations characterizing an optimizer  $(\xb^o,\pb^o,\ub^o,\rb^o)$:
\begin{equation}\label{eq:optimizer}
\left\lbrace\begin{array}{ll}
\dot{\xb}^o(t)=\mc{A}\xb^o(t)+\F(\xb^o(t))+\mc{B}(\rb^o)\ub^o(t),& \xb^o(0)=\xb_0,\\[2mm]
\dot{\pb}^o(t)=-(\mc{A}^*+\F'^*_{\xb^o(t)})\pb^o(t)-\mc{Q}\xb^o(t),& \pb^o(\tau)=0,\\[2mm]
\ub^o(t)=-\mc{R}^{-1}\mc{B}^*(\rb^o)\pb^o(t),\\[2mm]
\int_0^{\tau} (\mc{B}'_{\rb^o}\ub^o(t))^*\pb^o(t)\, dt=0.
\end{array}\right.
\end{equation}
If the control space $\mathbb U$ and actuator design space $\mathbb K$ are separable Hilbert spaces, the optimizing control and actuator can be characterized further.
Let $\eb^\mathbb{K}_j$, $\eb^\cs_i$, and $\eb^\ss_k$ be orthonormal bases for $\mathbb{K}$, $\cs$, and $\ss$, respectively. Then there exists $\bm{b}_i ( \rb) \in \ss$, $\rb\in \mathbb K$ so that for any $\ub \in \mathbb U$, 
\begin{equation}\label{eq:B}
\mc B (\rb) \ub=\sum_{i=1}^{\infty}  \inn{\ub}{\eb^\cs_i}_{\mathbb U}\bm{b}_i (\rb).
\end{equation}
Since the operator $\mc{B}(\cdot)\ub:\mathbb{K}\to \ss$ is {G\^ateaux} differentiable with respect to $\rb$, each $\bm{b}_i(\cdot)$ is a {G\^ateaux} differentiable map from $\mathbb{K}$ to $\ss$. Denote the {G\^ateaux} derivative of $\bm{b}_i(\rb)$ at $\rb^o$ by $\bm{b}'_{i,\rb^o}:\mathbb{K}\to \ss$, then
\begin{equation}\label{eq:Br}
(\mc B'_{\rb^o}\rb)\ub=\sum_{i=1}^{\infty}\sum_{j=1}^{\infty}  \inn{\ub}{\eb^\cs_i}_{\mathbb U}\inn{\rb}{\eb^\mathbb{K}_j}_\mathbb{K} \bm{b}'_{i,\rb^o}\eb^\mathbb{K}_j.
\end{equation}
\begin{corollary}\label{corollary2}
Assume further that the input space $\mathbb U$ and actuator design space $\mathbb K$ are separable. Let  $\eb^\cs_i$, $\eb^\mathbb{K}_j$ and $\eb^\ss_k$ be orthonormal bases for $\mathbb{K}$, $\cs$, and $\ss$, respectively. Define  $\ub^o_j(t)$ and $\pb_k(t)$ as
\begin{subequations}\label{eq:components}
\begin{flalign}
\ub^o_j(t)&\coloneqq\inn{\ub^o(t)}{\eb^\cs_j}_\cs,\\
\pb^o_k(t)&\coloneqq\inn{\pb^o(t)}{\eb^\ss_k}.
\end{flalign}%
\end{subequations}%
The optimality conditions (\ref{eq:optimality}) {in the interior of $U_{ad}\times K_{ad}$} can be written as
\begin{subequations}\label{eq:optimality components}
\begin{flalign}
\ub^o_j(t)+\sum_{i=1}^{\infty}\sum_{k=1}^{\infty}\inn{\bm{b}_i(\rb^o)}{\eb^\ss_k} \inn{\mc R^{-1} \eb_i^\cs}{\eb_j^\cs}_\cs\pb_k^o(t)&=0,\quad \text{for each}\; j, \\
\sum_{i=1}^{\infty}\sum_{k=1}^{\infty}\inn{\bm{b}'_{i,\rb^o}\eb_j^\mathbb{K}}{\eb_k^\ss} \int_0^{\tau} \ub^o_i (s)\pb^o_k(s)\, ds&=0,\quad  \text{for each}\; j .
\end{flalign}%
\end{subequations}%
\end{corollary}	
\begin{proof}
For every $\pb\in \ss$, the element $\mc B^*(\rb^o)\pb\in \cs$ can by obtained by using (\ref{eq:B}), and doing the calculation
\begin{flalign}
\inn{\mc{B}^*(\rb^o)\pb}{\ub}_\cs&=\inn{\pb}{\mc{B}(\rb^o)\ub}\notag\\
&=\sum_{i=1}^{\infty}  \inn{\ub}{\eb^\cs_i}_{\mathbb U} \inn{\pb}{\bm{b}_i (\rb^o)}\notag\\
&=\inn{\sum_{i=1}^{\infty}\inn{\bm{b}_i (\rb^o)}{\pb}\eb^\cs_i}{\ub}_\cs.
\end{flalign}
This yields
\begin{equation}
\mc{B}^*(\rb^o)\pb=\sum_{i=1}^{\infty}\inn{\bm{b}_i (\rb^o)}{\pb}\eb^\cs_i.
\end{equation}
Similarly, using  (\ref{eq:Br}), for every $\ub\in \cs$, the operator $(\mc{B}'_{\rb^o}\ub)^*$ maps $\pb \in \ss $ to $\mathbb K$ as follows
\begin{flalign}
(\mc B'_{\rb^o}\ub^o)^*\pb&=\sum_{j=1}^{\infty}\sum_{i=1}^{\infty}\inn{\ub^o}{\eb_i^\cs}_\cs \inn{\bm{b}'_{i,\rb^o}\eb_j^\mathbb{K}}{\pb} \eb^\mathbb{K}_j.
\end{flalign}
Substituting these elements into the optimality conditions (\ref{eq:optimality}) and using (\ref{eq:components}) leads to (\ref{eq:optimality components}). 
\end{proof}

\section{Railway Track Model}
Railway tracks are rested on ballast which are known for exhibiting nonlinear viscoelastic behavior \cite{ansari2011}. If a track beam is made of a Kelvin-Voigt material, then the railway track model will be a semi-linear partial differential equation on $\xi\in [0,\ell]$ as follows:
\begin{flalign}
&\rho a \frac{\partial^2 w}{\partial t^2}+\frac{\partial}{\partial \xi^2}(EI\frac{\partial^2 w}{\partial \xi^2}+C_d\frac{\partial ^3 w}{\partial \xi^2 \partial t})+\mu \frac{\partial w}{\partial t}+kw+\alpha w^3=b(\xi;r)u(t),\label{eq:pde parabolic}\\
&w(\xi,0)=w_0(\xi), \quad \frac{\partial w}{\partial t}(\xi,0)=v_0(\xi),\notag\\
&w(0,t)=w(\ell,t)=0,\notag\\
&EI\frac{\partial ^2w}{\partial \xi^2}(0,t)+C_d\frac{\partial ^3w}{\partial \xi^2\partial t}(0,t)=EI\frac{\partial^2 w}{\partial \xi^2}(\ell,t)+C_d\frac{\partial ^3w}{\partial \xi^2\partial t}(\ell,t)=0, \label{eq:KV}\notag
\end{flalign}
where the positive constants $E$, $I$, $\rho$, $a$, and $\ell$ are the modulus of elasticity, second moment of inertia, density of the beam, cross-sectional area, and length of the beam, respectively. The linear and nonlinear parts of the foundation elasticity correspond to the coefficients $k$ and $\alpha$, respectively. The constant $\mu\ge 0$ is the damping coefficient of the foundation, and $C_d\ge 0$ is the coefficient of Kelvin-Voigt damping in the beam.
{The track deflection is controlled by an external force  $u(t)$;  $u(t)$  will  be
assumed to be a scalar input in order to simplify the exposition. The shape influence function $b(\xi;r)$ is a continuous function over $[0,\ell]$ parametrized by the parameter $r$ that describes its dependence on actuator location. For example, as shown in \Cref{beam-schematic}, the control force is typically localized at some point $r$ and $b(\xi;r)$ models the distribution of the force $u(t)$ along the beam. The
function $b(\xi;  r) $  is assumed continuously differentiable with respect to $r$ over $\mathbb{R}$ (assumptions \ref{B}\ref{as:B2} and \ref{C}\ref{as:C2}); a suitable function for the case of actuator location is illustrated in \Cref{beam-schematic}.}

\begin{figure}[H]
\centering
\includegraphics[width=0.7\linewidth]{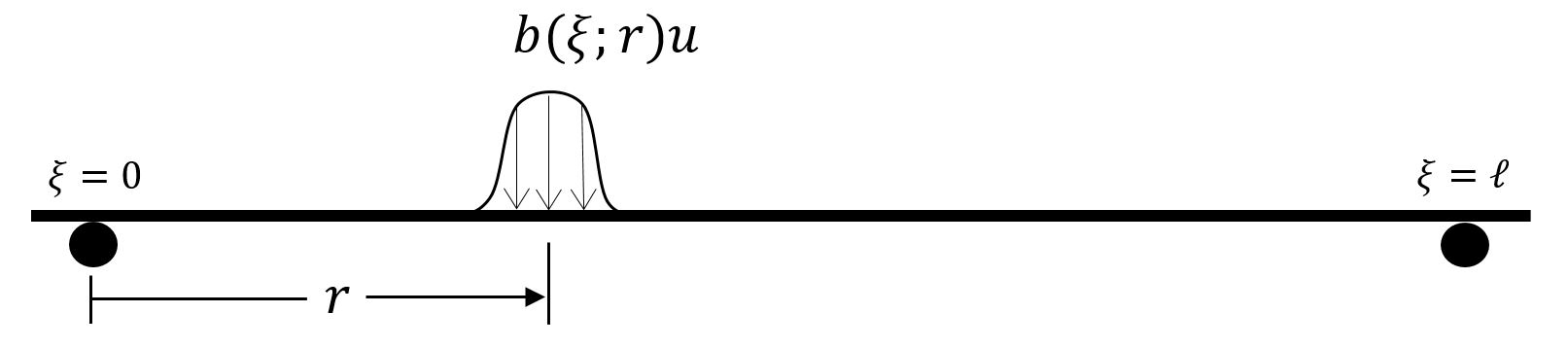}
\caption{Schematic of an actuator on the railway track beam.}
\label{beam-schematic}
\end{figure}

Define the closed self-adjoint positive operator $\mc{A}_0$ on $L^2(0,\ell)$ as:
\begin{flalign}
&\mc{A}_0w\coloneqq w_{\xi\xi\xi\xi},\notag \\
&D(\mc{A}_0)\coloneqq\left\lbrace w\in H^4(0,\ell )| \, w(0)=w(\ell)=0, \; w_{\xi\xi}(0)=w_{\xi\xi}(\ell)=0 \right\rbrace,
\end{flalign}
where {the subscript $\cdot_{\xi}$} denote the derivative with respect to the spatial variable. As a result, the state operator associated with the Kelvin-Voigt beam is 
\begin{equation}
\mc{A}_{\scriptscriptstyle KV}(w,v)\coloneqq\left(v,-\frac{1}{\rho a}\mc{A}_0(EIw+C_dv)\right),
\end{equation}
which is defined on the state space $\ss=H^2(0,\ell )\cap H_0^1(0,\ell )\times L^2(0,\ell )$ equipped with the norm
\begin{equation}
\| (w,v) \|^2=\int_0^{\ell} EIw_{\xi\xi}^2+kw^2+\rho a v^2 \, d\xi \label{eq: norm}.
\end{equation}
Accordingly, the domain of the state operator is 
\begin{equation}
D(\mc{A}_{\scriptscriptstyle KV})\coloneqq\left\lbrace(w,v)\in \ss| \, v\in H^2(0,\ell )\cap H_0^1(0,\ell ), \; EIw+C_dv\in D(\mc{A}_0) \right\rbrace.
\end{equation}
The underlying state space $\ss$ is separable since the spaces $H^2(0,\ell )\cap H_0^1(0,\ell )$ and $L^2(0,\ell )$ are separable. Furthermore, define the linear operators $\mc{K}$, $\mc{B}(r)$, and the nonlinear operator $\F(\cdot)$ as
\begin{flalign}
&\mc{K}(w,v)\coloneqq(0,-\frac{1}{\rho a}(\mu v + kw)),\\
&\mc{B}(r)u\coloneqq(0,\frac{1}{\rho a}b(\xi;r)u),\\
&\F(w,v)\coloneqq(0,-\frac{\alpha}{\rho a} w^3).\label{eq:cubic nonlinearity}
\end{flalign} 
The operator $\mc{K}$ is a bounded linear operator on $\ss$. For each $r$, operator $\mc{B}(r)$ is also a bounded operator that maps an input $u\in \mathbb{R}$ to the state space $\ss$. Since the space $H^2(0,\ell)$ is contained in the space of continuous functions over $[0,\ell]$, the the nonlinear term $w^3$ is in $L^2(0,\ell)$. Thus, the nonlinear operator $\F(\cdot)$ is well-defined on $\ss$. Lastly, define the operator $\mc{A}=\mc{A}_{\scriptscriptstyle KV}+\mc{K}$, with the same domain as $\mc{A}_{\scriptscriptstyle KV}$. With these definition and by setting $\xb=(w,v)$, the state space representation of the railway model (\ref{eq:KV}) is
\begin{equation}
\dot{\xb}(t)=\mc{A}\xb(t)+\F(\xb(t))+\mc{B}(r)u(t), \quad \xb(0)=\xb_0\in D(\mc{A}) \label{IVP2}.
\end{equation}
 It is straightforward to show that the operator $\mc{A}_0$ is closed, densely-defined, self-adjoint, and positive; it also has a compact resolvent. As a result, the operator $\mc{A}_{\scriptscriptstyle KV}$ will be a special case of the operator $\mathscr{A}_B$ in \cite{chen1989proof} with $\alpha=1$. According to Theorem 1.1 in \cite{chen1989proof}, such operators are generator of an analytic semigroup (also see \cite[Sec.~3]{banks1987unified} for a different approach). Furthermore, the operator $\mc{A}_{\scriptscriptstyle KV}+\mc{K}$ is a bounded perturbation of the operator $\mc{A}_{\scriptscriptstyle KV}$. By Corollary 3.2.2 in \cite{pazy}, $\mc{A}_{\scriptscriptstyle KV}+\mc{K}$ also generates an analytic semigroup. 

The railway track model in \cite{ansari2011} neglects the Kelvin-Voigt damping in the beam (i.e. $C_d=0$), and only includes Kelvin-Voigt damping in the ballast. In this case, the semigroup generated by $\mc{A}$ is not analytic. 
The results of this paper hold true for both models.

To guarantee the existence of a unique solution to the PDE (\ref{eq:pde parabolic}), the nonlinear operator $\F(\cdot)$ needs to fall into assumption \ref{A}\ref{as:A2}, \ref{B}\ref{as:B1}, \ref{C}\ref{as:C1}, and \ref{C}\ref{as:C2}. {The following result is due to Simon \cite[Thm. 3]{simon1986compact}, and will be used to check assumption \ref{B}\ref{as:B1}.
\begin{theorem}\cite[Thm. 3]{simon1986compact}\label{thm-simon}
Let $\ss$ and $\mathbb{Y}$ be Banach spaces, and $\mathbb{Y}\hookrightarrow \ss$ with compact embedding. Assume $X \subset L^p(0,\tau; \ss)$ where $1\le p \le \infty$, and
\begin{gather}
X \text{ is bounded in } L^1_{loc}(0,\tau;\mathbb{Y}),\\
\int_0^{\tau-h}\norm{\xb(t+h)-\xb(t)}{\ss}^pdt\to 0 \text{ as } h\to 0 \text{ uniformly for }\xb\in X.
\end{gather}
Then, $X$ is relatively compact in $L^p(0,\tau; \ss)$ (and in $C(0,\tau; \ss)$ if $p =\infty$).
\end{theorem}}
\begin{lemma} \label{lem:frechet F}The operator $\F(\cdot)$
\begin{enumerate} 
\item is continuously Fr\'echet differentiable on $\ss$; {the Fr\'echet derivative of this operator at $\xb=(w,v)$ maps every $\pb=(f,g)$ to $\F'_{\xb} \pb=(0,-3\alpha w^2f/\rho a)$,} 
\item {the mapping $\xb\mapsto \F'_{\xb}$ is bounded, and}
\item $\F(\cdot)$ satisfies assumption \ref{B}\ref{as:B1}.
\end{enumerate}
\end{lemma}
\begin{proof}
If $\F'_{\xb}$ is the Fr\'echet derivative of $\F(\cdot)$ at $\xb$, this operator must satisfy
\begin{equation}
\lim_{\pb\to 0}\frac{\|\F(\xb+\pb)-\F(\xb)-\F'_{\xb}\pb\|}{\|\pb\|}=0.
\end{equation}
Recall the definition of the operator $\F$ and that of norm on the space $\ss$, above limit simplifies to
\begin{equation}
\lim_{\|{f}\|_{H^2}\to 0}\frac{\|{f}^3+3{f}^2{w}\|_{L^2}}{\|{f}\|_{H^2}}=0. \label{eq: limit}
\end{equation}
Notice that functions $f$ and $w$ are in $H^2(0,\ell )$, and thus, continuous on $[0,\ell]$. Use triangle inequality, and H\"older's inequality to obtain
\begin{flalign}
\|{f}^3+3{f}^2{w}\|_{L^2} &\le \|{f}^3\|_{L^2}+\|3{f}^2{w}\|_{L^2}\notag\\
&\le \|{f}\|_{L^6}^3+3\|{f}\|_{L^8}^2\|{w}\|_{L^4}.
\end{flalign}
Apply the Sobolev embedding result $H^2(0,\ell )\hookrightarrow L^p(0,\ell )$ and let $c_p$ be the embedding constant
\begin{flalign}
\|{f}^3+3{f}^2{w}\|_{L^2} &\le c_6^3\|{f}\|_{H^2}^3+3c_8^2c_4\|{f}\|_{H^2}^2\|{w}\|_{H^2},
\end{flalign}
As a result, the expression in (\ref{eq: limit}) is bounded above according to
\begin{equation}
\frac{\|{f}^3+3{f}^2{w}\|_{L^2}}{\|{f}\|_{H^2}} \le c_6^3\|f\|^2_{H^2}+3c_8^2c_4\|w\|_{H^2}\|f\|_{H^2}.
\end{equation}
This shows that the limit in (\ref{eq: limit}) holds, and the operator $\F(\cdot)$ is indeed Fr\'echet differentiable.

Furthermore, select $\xb_1=(w_1,v_1)$, $\xb_2=(w_2,v_2)$, and $\pb=(f,g)$ as generic elements of $\ss$. The Fr\'echet derivative of $\F(\cdot)$ at $\xb_2-\xb_1$ is $$\F'_{\xb_2-\xb_1}\pb=(0,-\frac{3\alpha}{\rho a} (w_2-w_1)^2f).$$%
Take the norm of $\F'_{\xb_2-\xb_1}\pb$, and use H\"older's inequality to obtain
\begin{flalign}
\normm{\F'_{\xb_2-\xb_1}\pb}&=\frac{3\alpha}{\sqrt{\rho a}}\left(\int_0^{\ell} (w_2(\xi)-w_1(\xi))^4f^2(\xi)\,d\xi\right)^\frac{1}{2}\notag\\
&\le \frac{3\alpha}{\sqrt{\rho a}}\norm{w_2-w_1}{L^8}^2\norm{f}{L^4}.
\end{flalign}
Applying the Sobolev embedding result $H^2(0,\ell )\hookrightarrow L^p(0,\ell )$ yields
\begin{flalign}
\normm{\F'_{\xb_2-\xb_1}\pb}&\le \frac{3\alpha}{\sqrt{\rho a}}c_8^2c_4\norm{w_2-w_1}{H^2}^2\norm{f}{H^2}\notag\\
&\le \frac{3\alpha}{\sqrt{\rho a}}c_8^2c_4 \normm{\xb_2-\xb_1}^2\normm{\pb}.\label{eq15}
\end{flalign}
The last inequality indicates that the operator norm of $\F'_{\xb}$ continuously depends on $\xb$. 

{Inequality (\ref{eq15}) also yields 
\begin{equation}
\norm{\F'_{\xb}}{\mc L(\ss)}\le \frac{3\alpha}{\sqrt{\rho a}}c_8^2c_4 \normm{\xb}^2.
\end{equation}
This shows that the mapping $\xb\mapsto \F'_{\xb}$ is bounded.}

{To show that the nonlinear operator $\F(\cdot)$ satisfies assumption \ref{B}\ref{as:B1}, consider a bounded sequence $\xb_n(t)=(w_n(t),v_n(t))$ in $C(0,\tau;\ss)$ weakly converging to some element $\xb(t)=(w(t),v(t))$ in $L^p(0,\tau;\ss)$. 
It is shown that the sequence $w_n(t)$ satisfies conditions of \Cref{thm-simon}. The sequence $w_n(t)$ is by assumption bounded in $C(0,\tau;H^2(0,\ell)\cap H^1_0(0,\ell))$, and so in $C(0,\tau;L^6(0,\ell))$. This ensures that for all $p\in[1,\infty)$
\begin{equation}
\int_0^{\tau-h}\norm{w_n(t+h)-w_n(t)}{L^6(0,\ell)}^pdt\to 0 \text{ as } h\to 0 \text{ uniformly for all }n.
\end{equation}
Also, the space $H^2(0,\ell)\cap H^1_0(0,\ell)$ is compactly embedded in $L^6(0,\ell)$ by Rellich-Kondrachov compact embedding theorem \cite[Ch.~6]{adams2003sobolev}.
According to \Cref{thm-simon}, $w_n(t)$ has a strongly convergent subsequence in $L^p(0,\tau;L^6(0,\ell))$. Recall that $w_n(t)$ weakly converges to $w(t)$ in $L^2(0,\tau;H^2(0,\ell)\cap H^1_0(0,\ell))$ as well. A weak limit is unique; thus, $w_n\to w$ in $L^p(0,\tau;L^6(0,\ell))$. This further implies that $w_n^3\to w^3$ in $L^p(0,\tau;L^2(0,\ell))$. The nonlinear operator $\F(\cdot)$ maps $\xb_n(t)$ to
\begin{equation}
\F(w_n(t),v_n(t))=(0,\frac{\alpha}{\rho a}w_n^3(t)).
\end{equation}
Thus, the sequence $\F(w_n(t),v_n(t))$ strongly (and so weakly) converges to $\F(w(t),v(t))$ in $L^2(0,\tau;\ss)$.}
\end{proof}
The previous lemma ensures that the nonlinear operator $\F(\cdot)$ satisfies  assumption \ref{A}\ref{as:A2}. By \cref{thm:mild}, for control inputs $u\in L^p(0,\tau)$, $1<p<\infty$, the existence of a unique local mild solution is guaranteed.  

To address the optimization problem \ref{eq:optimal problem} for the railway track model, assumption \ref{B} and \ref{C} need to be satisfied. In \cref{lem:weak lower semi}, it was shown that assumption \ref{B}\ref{as:C3 B2} will hold for the particular choice of the cost function in assumption \ref{C}\ref{as:C3}.  As a result, the existence of an optimal pair $(u^o,r^o)$ together with an optimal trajectory $\xb^o$ follows from \cref{thm:existence optimizer}.

Accordingly, using \cref{thm:characterizing}, the optimal pair $(u^o,r^o)$ satisfies the equation (\ref{eq:optimizer}). In order to characterize the optimizers (\ref{eq:optimizer}) some adjoint operators need to be calculated. Calculation of the operator $\mc{A}^*$ is straightforward; it is 
\begin{equation}
\mc{A}^*(f,g)=\left(-g,\frac{1}{\rho a}\mc{A}_0(EI f-C_d g)+\frac{k}{\rho a}f-\frac{\mu}{\rho a}g\right), \
\end{equation}
for all  $(f,g)$ in the domain
\begin{equation}
D(\mc{A}^*)=\left\lbrace(f,g)\in \ss| \, g\in H^2(0,\ell )\cap H_0^1(0,\ell ), \; EIf-C_dg\in D(\mc{A}_0) \right\rbrace.
\end{equation}
Let $\xb^o(t)=(w^o,v^o)$ be the optimal trajectory evaluated at time $t\in [0,\tau]$. To calculate the adjoint of the operator $\F'_{\xb^o(t)}$ for every $t\in [0,\tau]$ on $\ss$, take the inner product $\F'_{\xb^o(t)}(w,v)$ with $(f,g) \in \ss$; that is
\begin{equation}\label{eq:adjoint}
\inn{\F'_{\xb^o(t)}(w,v)}{(f,g)}=\int_0^\ell -{3\alpha}(w^o(\xi))^2w(\xi)g(\xi)\, d\xi.
\end{equation}
For any $g \in L^2 (0,\ell)$, consider the function $h \in H^2 (0, \ell) \cap H^1_0 (0,\ell) $ satisfying the differential equation
\begin{flalign}\label{eq:phi}
&EI h_{\xi\xi\xi\xi}(\xi)+k h(\xi)=-{3\alpha}(w^o(\xi))^2g(\xi), \notag\\
&h(0)=h(\ell)=0,\notag \\
&h_{\xi\xi}(0)=h_{\xi\xi}(\ell)=0.
\end{flalign} 
An explicit solution $ h(\xi)$ to (\ref{eq:phi}) can be calculated using a Green's function:
\begin{flalign}
h(\xi)&=-{3\alpha}\int_0^{\ell} G(\xi,{\eta})(w^o({\eta}))^2g({\eta})\, d{\eta},\notag \\
G(\xi,{\eta})&=\frac{1}{6\ell}\left\{ \begin{array}{ll}
	(2\ell^2{\eta}-3\ell {\eta}^2+{\eta}^3)\xi+({\eta}-\ell)\xi^3,\quad \xi\le {\eta}\\
	({\eta}^3-\ell^2 {\eta})\xi+{\eta}\xi^3,\quad \xi>{\eta} \, . \\
	\end{array} \right.
\end{flalign}
With this calculation, for any $(w,v) \in \ss,$ 
\begin{flalign}
\inn{(w,v)}{(h,0)}&=\int_0^\ell EI w_{\xi\xi}(\xi)h_{\xi \xi}(\xi)+k w(\xi) h(\xi) \, d\xi \notag\\
&=EI[h_{\xi\xi}w_{\xi}]_0^\ell-EI[h_{\xi\xi\xi}w]_0^\ell+\int_0^\ell (EIh_{\xi\xi\xi\xi}(\xi)+kh(\xi))w(\xi)\, d\xi \notag\\
&=\int_0^\ell -{3\alpha}(w^o(\xi))^2w(\xi)g(\xi)\, d\xi.
\end{flalign}
Comparing this equation to (\ref{eq:adjoint}); the adjoint of $\F'_{\xb^o(t)}$ is defined by
\begin{equation}
\F'^*_{\xb^o(t)}(f,g)=(h,0).
\end{equation}

The adjoint of the operator $\mc{B}(r)$ for every $(f,g)\in\ss$ is
\begin{equation}
\mc{B}^*(r)(f,g)=\int_0^{\ell}b(\xi;r)g(\xi)\, d\xi.
\end{equation}
{Also, denote $b_{r}(\xi;r)$ to be the derivative of $b(\xi;r)$ with respect to $r$ and let $\pb^o(t)=(f,g)$ at time $t\in [0,\tau]$. Use \Cref{corollary2} to find
\begin{equation}
(\B'_ru)^*\pb^o(t)=u\int_0^\ell b_r(\xi;r)g(\xi) d\xi, \quad \forall (f,g)\in \ss.        
\end{equation}}

Furthermore, let $q_1\in C^2([0,\ell])$ and $q_2\in C([0,\ell])$ be some non-negative functions. Set $\mc{Q}(w,v)=(q_1 w,q_2 v)$ and $\mc{R}=1$ in the cost function of assumption \ref{C}\ref{as:C3}. 

In conclusion, {assuming that $(u^o,r^o)$ is in the interior of $U_{ad}\times K_{ad}$}, the following set of equations yields an optimizer for every initial condition $\xb_0=(w_0,v_0)\in \ss$:
\begin{flalign}\label{eq:optimizer for railway}
&\begin{cases}\tag{IVP}
\rho a w^o_{tt}+(EI w^o_{\xi\xi}+C_dw^o_{t\xi\xi})_{\xi\xi}+\mu w^o_t+kw^o+\alpha (w^o)^3=b(\xi;r^o)u^o(t),\\
w^o(0,t)=w^o(\ell,t)=0,\\
EIw^o_{\xi\xi}(0,t)+C_dw^o_{t\xi\xi}(0,t)=EIw^o_{\xi\xi}(\ell,t)+C_dw^o_{t\xi\xi}(\ell,t)=0\\
w^o(\xi,0)=w_0(\xi), \; w^o_t(\xi,0)=v_0(\xi),
\end{cases}\\
&\begin{cases}\tag{FVP}
\rho a f^o_{t}-\rho a g^o+3\alpha\int_0^\ell G(\xi,{\eta})(w^o({\eta}))^2g^o({\eta})\, d{\eta}=-\rho a q_1(\xi) w^o ,\\
\rho a g^o_{t}+(EI f^o_{\xi\xi}-C_dg^o_{\xi\xi})_{\xi\xi}-\mu g^o+kf^o=-\rho aq_2(\xi) w^o_t,\\
f^o(0,t)=f^o(\ell,t)=0,\; g^o(0,t)=g^o(\ell,t)=0,\\
EIf^o_{\xi\xi}(0,t)-C_dg^o_{\xi\xi}(0,t)=EIf^o_{\xi\xi}(\ell,t)-C_dg^o_{\xi\xi}(\ell,t)=0\\
f^o(\xi,T)=0, \; g^o(\xi,T)=0,
\end{cases}\\
&\begin{cases}\tag{OPT}
u^o(t)=-\int_0^{\ell}b(\xi;r^o)g^o(\xi,t)\, d\xi,\\
\int_0^{\tau}\int_0^{\ell}u^o(t)b_{r}(\xi;r^o)g^o(\xi,t)\,d\xi dt=0.
\end{cases}
\end{flalign}

\section{{Nonlinear Waves}}
Nonlinear waves occur in many applications, including  fluid mechanics, electromagnetism, elasticity, and also relativistic quantum mechanics.
Let the wave evolve on a region $\Omega$ that is a bounded, open, connected subset of $\mathbb{R}^2	$. It is assumed that $\Omega$ has a Lipschitz boundary  separated into $\partial \Omega=\overline{\Gamma_0\cup\Gamma_1}$ where  $\Gamma_0\cap\Gamma_1=\emptyset$ and $\Gamma_0\neq \emptyset$. Denote by $\nu$  the unit outward normal vector field on $\partial \Omega$. 
\Cref{wave-schematic} illustrates the region and the shape of an actuator.
Define $\as=L^2(\Omega)$ and let $r(\xi)$, $\xi\in \Omega$, be the actuator shape design. 
There are many possible choices of admissible shapes. One is
$$K_{ad}  = \{ r \in C^1 (\overline{\Omega} ) :  \norm{r}{C^1 (\overline{\Omega})}\leq 1 \} . $$ A nonlinear wave model with initial conditions $w_0(\xi)$ and $v_0(\xi)$ is 
\begin{equation}
\begin{cases}
\begin{aligned}
&\p{w}{t}{2}(\xi,t)=\Delta w(\xi,t) +F(w(\xi,t))+r(\xi)u(t), &(\xi,t)&\in \Omega\times (0,\infty),\\
&w(\xi,0)=w_0(\xi), \; \p{w}{t}{}(\xi,0)=v_0(\xi),  &\xi&\in \Omega,\notag\\
&w(\xi,t)=0, &(\xi,t)&\in \Gamma_0\times [0,\infty),\notag\\
&\p{w}{\nu}{}(\xi,t)=0, &(\xi,t) &\in \Gamma_1\times [0,\infty).\notag
\end{aligned}
\end{cases}
\end{equation}

\begin{figure}[H]
\centering
\includegraphics[width=0.7\linewidth]{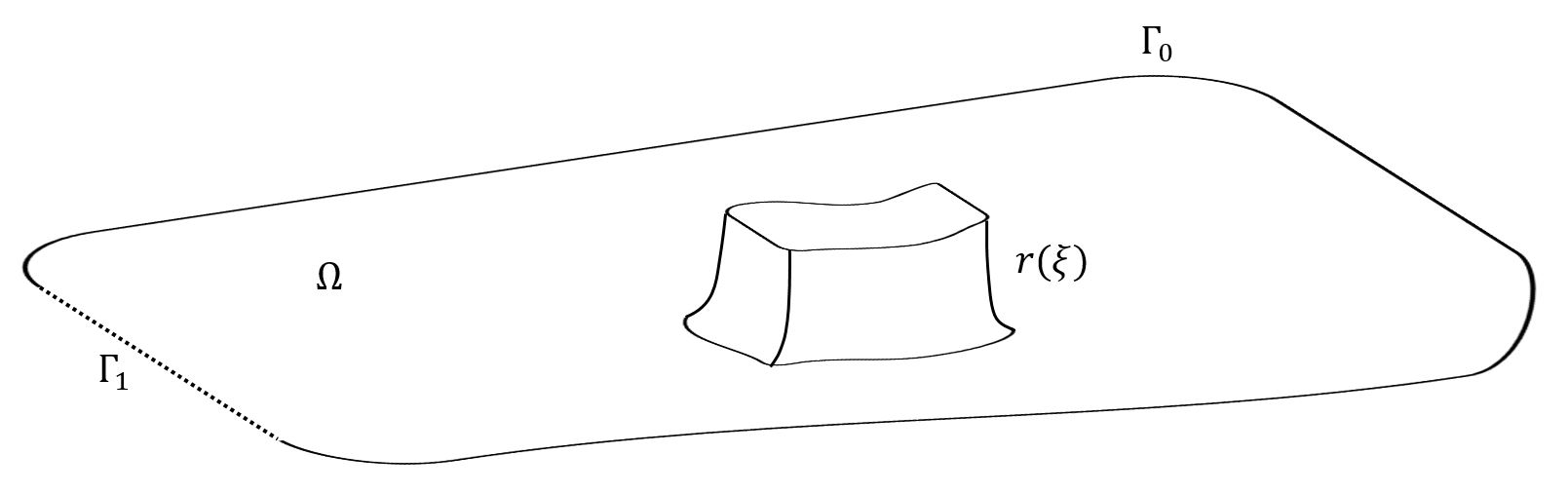}
\caption{Schematic of an actuator on the wave region.}
\label{wave-schematic}
\end{figure}

Let $\ss=H^1_{\Gamma_0}(\Omega)\times L^2(\Omega)$ and define $\A:D(\A)\to \ss$ as
\begin{gather}
\A(w,v)=(v,\nabla w),\\
D(\A)=\left\{(w,v)\in \ss\, | \, w\in H^2(\Omega)\cap H_{\Gamma_0}^1(\Omega), \; v\in H^1_{\Gamma_0}(\Omega),\; \p{w}{\nu}{}{\Big |}_{\Gamma_1}=0 \right\}.\notag
\end{gather}
The operator $\A$ is skew-adjoint and generates a strongly continuous unitary group on $\ss$; see for example, \cite[Thm.~3.24]{engel2000}.

\begin{assumption}\label{D}\leavevmode
\begin{enumerate}
\item \label{D1}The function $F(\zeta)$ is twice continuously differentiable over $\mathbb{R}$; denote its derivatives by $F'(\zeta)$ and $F'' (\zeta)$.
\item \label{D2}There are numbers $a_0>0$ and $b>1/2$ such that $|F''(\zeta)|\le a_0(1+|\zeta|^b)$.
\end{enumerate}
\end{assumption}
The nonlinear operator $\F(\cdot):\ss\to \ss$ is  defined as 
\begin{equation}\label{F-wave}
\F(w,v) =(0,F(w) ) .
\end{equation}
Assumption \ref{D} is needed to ensure that $\F (\cdot) : \ss \to \ss $ and satisfies assumption \ref{B}\ref{as:B1} and  that  the G\^ateaux derivative of $\F(\cdot)$ is also an operator on $\ss$.
Some examples of $F(\cdot)$ satisfying this assumption are $F(w)=\sin(w)$ in the  Sine-Gordon equation and $F(w)=|w|^kw, \; k\ge2$ in the Klein-Gordon equation \cite[Sec.~5.2]{sell2013dynamics}.

\begin{lemma}
Under assumption \ref{D}, 
\begin{enumerate}
\item the operator $\F(\cdot)$   is G\^ateaux differentiable on $\ss$, with the G\^ateaux derivative at $\xb=(w,v)$ in the direction $\pb=(f,g)$ given by $\F'_{\xb}\pb=(0,F'(w)f)$,
\item {the mapping $\xb\mapsto \F'_{\xb}$ is bounded, and}
\item $\F(\cdot)$ satisfies assumption \ref{B}\ref{as:B1}.
\end{enumerate}
\end{lemma}
\begin{proof}
To prove the first part of the lemma, it must be shown that for any variation $f \in  H^1_{\Gamma_0}(\Omega) ,$
\begin{equation}
\lim_{ \epsilon\to 0}\norm{\frac{1}{\epsilon}\left(F(w+\epsilon f)-F(w)\right)-F'(w)f}{L^2(\Omega)}=0,
\end{equation}
or 
\begin{equation}\label{eq5}
\lim_{ \epsilon\to 0}\int_{\Omega}{\big |}\frac{1}{\epsilon}\left(F(w(\xi)+\epsilon f(\xi))-F(w(\xi))\right)-F'(w(\xi))f(\xi){\big |}^2 \, d\xi=0.
\end{equation}
Recall that because of the continuous embedding $H^1_{\Gamma_0}(\Omega)\hookrightarrow L^p(\Omega)$, the functions $f$ and $w$ belong to $L^p(\Omega)$ for all $p\in [1,\infty)$. 
Use of assumption \ref{D}, applying  Taylor's theorem with integral reminder  to $F(\cdot)$, and using Jensen's inequality, the integral in (\ref{eq5}) becomes
\begin{flalign}
&\int_{\Omega}\left(\int_0^1\epsilon(1-\eta) F''(w(\xi)+\eta\epsilon f(\xi))f^2(\xi)d\eta\right)^2 d\xi\notag\\
&\le \int_{\Omega}\int_0^1\epsilon^2(1-\eta)^2 F''^2(w(\xi)+\eta\epsilon f(\xi))f^4(\xi)d\eta d\xi \notag \\
&\le \int_{\Omega}\int_0^1\epsilon^2(1-\eta)^2 a_0^2\left(2+2^{2b}|w(\xi)|^{2b}+2^{2b}\eta^{2b}|\epsilon|^{2b} |f(\xi)|^{2b}\right)f^4(\xi)d\eta d\xi \label{eq6}
\end{flalign}
Applying H\"older's inequality shows that integral (\ref{eq6}) is bounded above by a number, and also converges to zero as $\epsilon\to 0$.

Furthermore,  the operator $\F'_{\xb}$ satisfies, for any $\xb=(w,v)$ and $\pb=(f,g)$ in $\ss$,
\begin{equation}
\normm{\F'_{\xb}\pb}^2=\int_{\Omega}F'^2(w(\xi))f^2(\xi)d\xi.
\end{equation}
Assumption \ref{D}\ref{D1} ensures that there is a number $a_1>0$ such that $|F'(\zeta)|\le a_1(1+|\zeta|^{b+1})$. Use this together with H\"older's inequality to obtain
\begin{flalign}
\normm{\F'_{\xb}\pb}^2&\le \int_{\Omega}2a_1^2(1+w^{2b+2}(\xi))f^2(\xi)d\xi.\notag\\
&\le 2a^2_1\left(\norm{f}{L^2(\Omega)}^2+\norm{w}{L^{4b+4}(\Omega)}^{2b+2}\norm{f}{L^4(\Omega)}^2\right)
\end{flalign}
Apply the embeddings $H^1_{\Gamma_0}(\Omega)\hookrightarrow L^{p}(\Omega)$; letting  $c_{p}$ indicates the various embedding constants,
\begin{flalign}
\normm{\F'_{\xb}\pb}^2 &\le 2a^2_1\left(c^2_2+c_{4b+4}^{2b+2}c_4^2\norm{w}{H^1_{\Gamma_0}(\Omega)}^{2b}\right)\norm{f}{H^1_{\Gamma_0}(\Omega)}^2\notag\\
&\le 2a^2_1\left(c^2_2+c_{4b+4}^{2b+2}c_4^2\normm{\xb}^{2b}\right)\normm{\pb}^2.\label{eq16}
\end{flalign}
Inequality (\ref{eq16}) implies that
\begin{equation}
\norm{\F'_{\xb}}{\mc{L}(\ss)}^2\le 2a^2_1\left(c^2_2+c_{4b+4}^{2b+2}c_4^2\normm{\xb}^{2b}\right).
\end{equation}
This inequality shows that the mapping $\xb\mapsto\F'_{\xb}$ is bounded.

It will now be shown that   $\F(\cdot)$ satisfies assumption \ref{B}\ref{as:B1}.  Consider a bounded sequence $\xb_n(t)=(w_n(t),v_n(t))$ in $C(0,\tau;\ss)$ that weakly converges to some element $\xb(t)=(w(t),v(t))$ in $L^p(0,\tau;\ss)$. The sequence $w_n(t)$ is bounded in $C(0,\tau;H^1_{\Gamma_0}(\Omega))$ and so  it is in $C(0,\tau;L^q(\Omega))$ for all $q\in [1,\infty)$. This together with the bounded convergence theorem ensures that for every $p\in[1,\infty)$
\begin{equation}\label{eq7}
\int_0^{\tau-h}\norm{w_n(t+h)-w_n(t)}{L^q(\Omega)}^pdt\to 0 \text{ as } h\to 0 \text{ uniformly for all }n.
\end{equation}
The space $H^1_{\Gamma_0}(\Omega)$ is compactly embedded in $L^q(\Omega)$ by Rellich-Kondrachov compact embedding theorem \cite[Ch.~6]{adams2003sobolev}. By \Cref{thm-simon}, this embedding together with (\ref{eq7}) ensures that $w_n(t)$ has a strongly convergent subsequence in $L^p(0,\tau;L^q(\Omega))$. The sequence $w_n(t)$ by assumption converges weakly to $w(t)$ in $L^p(0,\tau;H^1_{\Gamma_0}(\Omega))$; a weak limit is unique, so $w_n(t)$ converges strongly to $w(t)$ in $L^p(0,\tau;L^q(\Omega))$. The nonlinear operator $\F(\cdot)$ maps $\xb_n(t)$ to
\begin{equation}
\F(w_n(t),v_n(t))=(0,F(w_n(t))).
\end{equation}
Use Taylor's theorem with integral reminder, and let $h(\xi,t;\eta)=w(\xi,t)+\eta(w_n(\xi,t)-w(\xi,t))$, $\eta\in [0,1]$, to obtain
\begin{flalign}
|F(w_n(\xi,t))-F(w(\xi,t))|&\le \left(\int_0^1|F'(h(\xi,t;\eta))|d\eta\right) | w_n(\xi,t)-w(\xi,t)|. \label{eq11}
\end{flalign}
Let $$M_1(\xi,t)=a_1 \left(1+\int_0^1|h(\xi,t;\eta)|^{b+1} d\eta\right).$$
Taking integral of both side of (\ref{eq11}) and using H\"older inequality yield
\begin{flalign}
\norm{F(w_n)-F(w)}{L^p(0,\tau;L^2(\Omega))}^p&\le\int_0^\tau\left(\int_\Omega M_1^2(\xi,t) |w_n(\xi,t)-w(\xi,t)|^2d\xi\right)^{\frac{p}{2}}dt\notag \\
&\le \norm{M_1}{L^{2p}(0,\tau;L^4(\Omega))}^p \norm{w_n-w}{L^{2p}(0,\tau;L^4(\Omega))}^p.\label{eq13}
\end{flalign}
Note that $\norm{M_1}{L^{2p}(0,\tau;L^4(\Omega))}< \infty$ since $w(\xi,t)$ and $w_n(\xi,t)$ are in $L^p(0,\tau;L^q(\Omega))$ for all $p\in [1,\infty)$ and $q\in [1,\infty)$.
From (\ref{eq13}), it follows that $F(w_n)$ strongly converges to $F(w)$ in $L^p(0,\tau;L^2(\Omega))$. Therefore, the sequence $\F(w_n(t),v_n(t))$ strongly (and so weakly) converges to $\F(w(t),v(t))$ in $L^p(0,\tau;\ss)$.
\end{proof}

Let $\xb^o(t)=(w^o,v^o)$ at time $t\in [0,\tau]$. As for the railway track example, in  order to obtain an expression for the adjoint of the operator $\F'_{\xb^o(t)}$, the following boundary-value problem needs to be solved:
\begin{equation}
\begin{cases}
\begin{aligned}
&\Delta h(\xi)=-F'(w^o(\xi))g(\xi),\quad &\xi& \in \Omega,\\
&h(\xi)=0, \quad &\xi& \in \Gamma_0,\\
&\p{h}{\nu}{}(\xi)=0, \quad &\xi& \in \Gamma_1.
\end{aligned}
\end{cases}
\label{eq-adjoint-bvp-wave}
\end{equation}
The adjoint with respect to  $\ss$ of  $\F'_{\xb^o(t)}$  is
\begin{equation}
\F'^*_{\xb^o(t)}(f,g)=(h,0),
\end{equation}
where $h$ solves \eqref{eq-adjoint-bvp-wave}.
Define $\cs =\mathbb{R}$ and the input operator $\B(r)\in \mc{L}(\cs ,\ss)$ by
\begin{equation}
\B(r)u=(0,r(\xi)u).
\end{equation} 
The adjoint of this operator is
\begin{equation}
\B^*(r)(f,g)=\int_{\Omega}r(\xi)g(\xi)d \xi, \quad \forall (f,g)\in \ss.
\end{equation}
Let $\pb^o(t)=(f,g)$ at time $t\in [0,\tau]$. Use \Cref{corollary2} to find
\begin{equation}
(\B'_ru)^*\pb^o(t)=ug. 
\end{equation}

Furthermore, let $q_1\in C^1(\overline{\Omega})$ and $q_2\in C(\overline{\Omega})$ be some non-negative functions. Set $\mc{Q}(w,v)=(q_1 w,q_2 v)$ and $\mc{R}=1$ in the cost function of assumption \ref{C}\ref{as:C3}.

If the optimal control $u^o$ and optimal actuator design $r^o$ are in the interior of $U_{ad}\times K_{ad}$, then by \Cref{corollary2} the following  equations are satisfied:
\begin{flalign}
&\begin{cases}
\begin{aligned}
&\p{w^o}{t}{2}(\xi,t)=\Delta w^o(\xi,t) +F(w^o(\xi,t))+r^o(\xi)u^o(t), &(\xi,t)&\in \Omega\times (0,\tau],\\
&w^o(\xi,0)=w_0(\xi), \; \p{w^o}{t}{}(\xi,0)=v_0(\xi),  &\xi&\in \Omega,\notag\\
&w^o(\xi,t)=0, &(\xi,t)&\in \Gamma_0\times [0,\tau],\notag\\
&\p{w^o}{\nu}{}(\xi,t)=0, &(\xi,t) &\in \Gamma_1\times [0,\tau],\notag
\end{aligned}
\end{cases}\tag{IVP}\\
&\begin{cases}
\begin{aligned}
&\p{f^o}{t}{}(\xi,t)=-g^o(\xi,t) -h^o(\xi,t)-q_1(\xi)w^o(\xi,t), &(\xi,t)&\in \Omega\times (0,\tau],\\[1mm]
&\p{g^o}{t}{}(\xi,t)=-\Delta f^o(\xi,t)-q_2(\xi)\p{w^o}{t}{}(\xi,t), &(\xi,t)&\in \Omega\times (0,\tau],\\[1mm]
&f^o(\xi,\tau)=0, \; g^o(\xi,\tau)=0,  &\xi&\in \Omega,\notag\\[1mm]
&f^o(\xi,t)=0, &(\xi,t)&\in \Gamma_0\times [0,\tau],\notag\\[1mm]
&\p{f^o}{\nu}{}(\xi,t)=0, &(\xi,t) &\in \Gamma_1\times [0,\tau],\notag
\end{aligned}
\end{cases}\tag{FVP}\\
&\begin{cases}
\begin{aligned}
&u^o(t)=-\int_{\Omega}r^o(\xi)g^o(\xi,t)d\xi, &t&\in [0,\tau],\\
&\int_0^\tau u^o(t)g^o(\xi,t) dt=0, &\xi&\in \Omega. 
\end{aligned}
\end{cases}\tag{OPT}
\end{flalign}

\section{Conclusions}
\label{sec:conclusions}
Optimal control of semi-linear infinite-dimensional systems was considered in this paper where the optimal controller design involves both the controlled input and the actuator design.  It was shown that the existence of an optimal control together with an optimal actuator design is guaranteed under some assumptions. Moreover, first-order necessary optimality conditions were obtained. The theory was illustrated with several applications. 

Current work is concerned with developing numerical methods for solution of the optimality equations and also the consideration of a wider class of nonlinearities. {Extension of these problems to situations  where the input operator is not bounded on the state space is also of interest.}

\appendix
 
\section{Proof of \cref{pro:lipschitz S}} \label{ap:appendix a}
\setcounter{equation}{0}
\renewcommand{\theequation}{\thesection.\arabic{equation}}
For $\xb_0\in \ss$ and $\rb\in K_{ad}$, consider $\xb_1(t)$ and $\xb_2(t)$ as the mild solutions to (\ref{IVP}) corresponding to the inputs $\ub_1(t)$ and $\ub_2(t)$, respectively. The inputs are in a ball of radius $R$ contained in $L^p(0,\tau;\cs)$, $1<p<\infty$; consequently, by \cref{corollary1} and assumption \ref{A}\ref{as:A3}, the states $\xb_1(t)$ and $\xb_2(t)$ are contained in a ball of radius
\begin{equation}
\delta=c_{\tau}(\normm{\xb_0}+M_\mc{B} R).
\end{equation}
From (\ref{eq:mild solution}), it follows that
\begin{flalign}
\xb_2(t)-\xb_1(t)=&\int_0^t \mc{T}(t-s)\left(\F(\xb_2(s))-\F(\xb_1(s))\right) ds\notag\\
&+\int_0^t \mc{T}(t-s) \mc{B}(\rb)\left(\ub_2(s)-\ub_1(s)\right) ds.	
\end{flalign} 
Recall that $\mc{T}(t)$ satisfies $\normm{\mc{T}(t)}\le M_{\T}$ for all $t\in [0,\tau]$ and some number $M_{\T}>0$. Also, remember that the operator $\F(\cdot)$ is locally Lipschitz continuous, and $\mc{B}(\rb)$ is uniformly bounded in $\ss$ for all $\rb\in K_{ad}$. Taking the norm in $\ss$ of both sides of this equation yields
\begin{flalign}
\normm{\xb_2(t)-\xb_1(t)}\le& M_{\T}L_{\F\delta}\int_0^t\normm{\xb_2(s)-\xb_1(s)}\, ds\notag\\
&+M_{\T}M_\mc{B}\tau^{(p-1)/p} \norm{\ub_2-\ub_1}{p}.
\end{flalign}
Define the constant $L_{\ub}$ as
\begin{equation}
L_{\ub}=e^{M_{\T}L_{\F\delta}\tau}M_{\T}M_\mc{B}\tau^{(p-1)/p}.
\end{equation}
By Gronwall's Lemma \cite[Thm.~1.4.1]{zettl2005}, it follows that
\begin{equation}
\norm{\xb_2-\xb_1}{C(0,\tau;\ss)}\le L_{\ub} \norm{\ub_2-\ub_1}{p}.
\end{equation} 
This is in fact the inequality (\ref{eq:lip in u}).

Similarly, for a fixed initial condition $\xb_0\in \ss$ and control input $\ub\in U_{ad}$, consider $\xb_1(t)$ and $\xb_2(t)$ as the mild solutions to (\ref{IVP}) corresponding to the actuator designs $\rb_1$ and $\rb_2$, respectively. Use local Lipschitz continuity of $\F(\cdot)$ and growth condition on semigroup $\mc{T}(t)$ and obtain
\begin{flalign}\label{eq:c5}
\normm{\xb_2(t)-\xb_1(t)}\le& M_{\T}L_{\F\delta}\int_0^t\normm{\xb_2(s)-\xb_1(s)}\, ds\notag\\
&+M_{\T}\tau^{(p-1)/p}\norm{\ub}{p}\norm{\mc{B}(\rb_2)-\mc{B}(\rb_1)}{\mc{L}(\cs,\ss)}.
\end{flalign}%
Assumption \ref{C}\ref{as:C2} implies that the control operator $\mc{B}(\rb)$ is locally Lipschitz continuous with respect to $\rb$. That is, letting $$L_\mc{B}=\sup\{\norm{\B'_\rb}{\mc L(\as,\mc L(\cs,\ss))}:\rb \in K_{ad}\},$$ operator $\B(\rb)$ for all $\rb_1$ and $\rb_2$ in $K_{ad}$ satisfies
\begin{equation}
\norm{\mc{B}(\rb_2)-\mc{B}(\rb_1)}{\mc{L}(\cs,\ss)}\le L_\mc{B} \norm{\rb_2-\rb_1}{\mathbb{K}}.
\end{equation} 
Accordingly, the inequality (\ref{eq:c5}) can be re-written as
\begin{flalign}
\normm{\xb_2(t)-\xb_1(t)}\le& M_{\T}L_{\F\delta}\int_0^t\normm{\xb_2(s)-\xb_1(s)}\, ds\notag\\
&+M_{\T}\tau^{(p-1)/p}R L_{\mc{B}}\norm{\rb_2-\rb_1}{\mathbb{K}}.
\end{flalign}
Denote the constant $L_{\rb}$ by
\begin{equation}
L_{\rb}=e^{M_{\T}L_{\F\delta}\tau}M_{\T}\tau^{(p-1)/p}R L_{\mc{B}}.
\end{equation}
Use Gronwall's Lemma \cite[Thm.~1.4.1]{zettl2005} to derive
\begin{equation}
\norm{\xb_2-\xb_1}{C(0,\tau;\ss)}\le L_{\rb} \norm{\rb_2-\rb_1}{\mathbb{K}}. \label{eq:d11}
\end{equation} 
This is in fact the inequality (\ref{eq:lip in r}).

\section{Proof of \cref{pro:frechet derivative in r}} 
\setcounter{equation}{0}
\label{ap:appendix d} Let number $\epsilon$ be small enough so that $\rb^o+\epsilon\tilde{\rb}\in K_{ad}$. Denote by $\xb=\mc{S}(\ub;\rb^o+\epsilon \tilde{\rb},\xb_0)$ the mild solution to the IVP
\begin{equation}
\dot{\xb}(t)=\mc{A}\xb(t)+\F(\xb(t))+\mc{B}(\rb^o+\epsilon \tilde{\rb})\ub(t), \quad \xb(0)=\xb_0.\label{eq:yp}
\end{equation}
The state $\xb^o=\mc{S}(\ub;\rb^o,\xb_0)$ is the mild solution of the IVP 
\begin{equation}
\dot{\xb}^o(t)=\mc{A}\xb^o(t)+\F(\xb^o(t))+\mc{B}(\rb^o)\ub(t), \quad \xb^o(0)=\xb_0. \label{eq:ystar}
\end{equation}
Define $\xb_{e}=(\xb-\xb^o)/\epsilon-\tilde{\yb}$,
subtract the equations (\ref{eq:ystar}) and (\ref{eq:ytilde}) from (\ref{eq:yp}), obtain 
\begin{flalign}\label{eq:c8}
\dot{\xb}_{e}(t)=&(\mc{A}+\F'_{\xb^o(t)})\xb_{e}(t)+\frac{1}{\epsilon}\left(\F(\xb(t))-\F(\xb^o(t))-\F'_{\xb^o(t)}(\xb(t)-\xb^o(t))\right)\notag \\
&+\left(\frac{1}{\epsilon}\left(\B(\rb^o+\epsilon\tilde{\rb})-\B(\rb^o)\right)-\B'_{\rb^o}\tilde{\rb}\right)\ub(t), \quad \xb_{e}(0)=0.
\end{flalign}
Define $\bm{\eb}_{\F}(t)$ and $\bm{\eb}_{\B}$ as 
\begin{subequations}
\begin{flalign}
\bm{\eb}_{\F}(t)&\coloneqq \frac{1}{\epsilon}\left(\F(\xb(t))-\F(\xb^o(t))-\F'_{\xb^o(t)}(\xb(t)-\xb^o(t))\right),\\
\bm{\eb}_{\B} &\coloneqq \frac{1}{\epsilon}\left(\B(\rb^o+\epsilon\tilde{\rb})-\B(\rb^o)\right)-\B'_{\rb^o}\tilde{\rb}.
\end{flalign}
\end{subequations}%
Assumption \ref{C}\ref{as:C1} and \ref{C}\ref{as:C2} ensure that as $\epsilon\to 0$
\begin{subequations}
\begin{gather}
\normm{\bm{\eb}_{\F}(t)}\to 0, \quad \forall t\in [0,\tau],\label{e(t)2}\\
\norm{\bm{\eb}_{\B}}{\mc{L}(\cs,\ss)} \to 0.\label{eB}
\end{gather}\label{eq:d13}%
\end{subequations}%
Also, similar to inequality (\ref{bound on e(t)}), using \Cref{pro:lipschitz S}\ref{-b}, and letting $\delta=c_\tau(\normm{\xb_0}+M_{\B}R)$ and $M_{\F'}=\sup \{\|\F'_{\xb^o(t)}\|: \;t\in[0,\tau]\}$, the following upper bounded can be obtained
\begin{equation}\label{bound on e(t)2}
\normm{\eb_\F(t)}\le \left(L_{\F\delta}+M_{\F'}\right)L_\rb\norm{\tilde{\rb}}{\as}, \quad \forall t\in [0,\tau].
\end{equation}

Rewrite (\ref{eq:c8}) as follows
\begin{flalign}
\dot{\xb}_{e}(t)=&(\mc{A}+\F'_{\xb^o(t)})\xb_{e}(t)+\bm{\eb}_{\F}(t)+\bm{\eb}_{\B}\ub(t), \quad \xb_e(0)=0.
\end{flalign}
According to \cref{lem:adjoint}\ref{a}, the mild solution of this evolution equation is described by an evolution operator $\mc{U}(t,s)$ as follows
\begin{equation}
\dot{\xb}_{e}(t)=\int_0^t\mc{U}(t,s)\eb_{\F}(s)ds+\int_0^t\mc{U}(t,s)\bm{\eb}_{\B}\ub(s)ds.
\end{equation}
 Let $M_{\mc{U}}$ be an upper bound for the operator norm of $\mc{U}(t,s)$ over $0\le t\le s \le \tau$,
\begin{flalign}
\norm{\xb_{e}}{L^p(0,\tau;\ss)}&\le \tau^{1/p}\norm{\xb_e}{C(0,\tau;\ss)}\notag \\
&\le \tau^{1/p}M_{\mc{U}}\int_0^{\tau} \normm{\bm{\eb}_{\F}(t)}+\tau^{1/p}M_{\mc{U}}\norm{\bm{\eb}_{\B}}{\mc{L}(\cs , \ss)}\norm{\ub}{1}\label{eq:d3}.
\end{flalign}
As a result of (\ref{e(t)2}) and (\ref{bound on e(t)2}), the first integral in (\ref{eq:d3}) tends to zero by the bounded convergence theorem. The second term of (\ref{eq:d3}) also converges to zero using (\ref{eB}). It follows that
\begin{equation}
\lim_{\epsilon\to 0}\norm{\frac{1}{\epsilon}\left(\mc{S}(\ub;\rb^o+\epsilon\tilde{\rb},\xb_0)-\mc{S}(\ub;\rb^o,\xb_0)\right)-\mc{S}'_{\rb^o}{\tilde{\rb}}}{L^p(0,\tau;\ss)}=\lim_{\epsilon\to 0}\norm{\xb_{e}}{L^p(0,\tau;\ss)}=0.\notag
\end{equation}
This shows that $\mc{S}'_{\rb^o}{\tilde{\rb}}$ is the G\^ateaux derivative of $\mc{S}(\ub;\rb,\xb_0)$ at $\rb^o$ in the direction $\tilde{\rb}$.

\bibliographystyle{siamplain}
\bibliography{library}
\end{document}